\documentclass[12pt]{amsart}

\usepackage{amssymb,latexsym}

\usepackage{enumerate}

\makeatletter

\@namedef{subjclassname@2010}{

  \textup{2010} Mathematics Subject Classification}

\makeatother

\newtheorem{normapprox}{Normal Approximation Result}

\newtheorem{normcomp1}{Normal Comparison Result}
\newtheorem{normcomp2}[normcomp1]{Normal Comparison Result}
\newtheorem{normcomp3}[normcomp1]{Normal Comparison Result}

\newtheorem{correst1}{Correlation Estimate}

\newtheorem{thm}{Theorem}[section]
\newtheorem{Con}[thm]{Conjecture}

\newtheorem{lem}[thm]{Lemma}
\newtheorem{pro}[thm]{Proposition}
\theoremstyle{definition}

\numberwithin{equation}{section}

\newcommand{\pr}{\mathbb{P}}
\newcommand{\ex}{\mathbb{E}}
\newcommand{\re}{\textup{Re}}

\newcommand{\R}{\mathbb{R}}
\newcommand{\E}{\mathbb{E}}
\newcommand{\p}{\mathbb{P}}
\newcommand{\Z}{\mathbb{Z}}
\newcommand{\C}{\mathbb{C}}
\newcommand{\var}{\mathrm{Var}}

\newcommand{\Q}{\mathbb{Q}}

\begin{document}

\baselineskip=17pt

\title[Weakly correlated random variables and prime number races]{Orderings of weakly correlated random variables, and prime number races with many contestants}

\author{Adam J Harper}
\address{Jesus College, Cambridge, CB5 8BL, England}
\email{A.J.Harper@dpmms.cam.ac.uk}
\thanks{AJH is supported by a research fellowship at Jesus College, Cambridge; and, when this work was started, by a postdoctoral fellowship from the Centre de Recherches Math\'{e}matiques, Montr\'{e}al. YL  is partially supported by a Discovery Grant from the Natural Sciences and Engineering Research Council of Canada}

\author{Youness Lamzouri}
\address{Department of Mathematics and Statistics, York University,
4700 Keele Sreet, Toronto,
 ON M3J 1P3, Canada}
\email{lamzouri@mathstat.yorku.edu}

\date{}

\begin{abstract}
We investigate the race between prime numbers in many residue classes modulo $q$, assuming the standard conjectures GRH and LI.

Among our results we exhibit, for the first time, prime races modulo $q$ with $n$ competitor classes where the biases do not dissolve when $n, q\to \infty$. We also study the leaders in the prime number race, obtaining asymptotic formulae for logarithmic densities when the number of competitors can be as large as a power of $q$, whereas previous methods could only allow a power of $\log q$.

The proofs use harmonic analysis related to the Hardy--Littlewood circle method to control the average size of correlations in prime number races. They also use various probabilistic tools, including an exchangeable pairs version of Stein's method, normal comparison tools, and conditioning arguments. In the process we derive some general results about orderings of weakly correlated random variables, which may be of independent interest.
\end{abstract}

\maketitle

\section{Introduction}

In an 1853 letter to Fuss, Chebyshev noted that on a fine scale there seem to be more primes congruent to $3$ than to $1$ modulo $4$. This observation led to the birth of comparative prime number theory,  which investigates the discrepancies in the distribution of prime numbers. A central problem is the so-called ``Shanks--R\'enyi prime number race'' which is described by Knapowski and Tur\'an~\cite{KT}: let $q\geq 3$ and $2\leq n\leq \varphi(q)$ be positive integers, (where the Euler function $\varphi(q)$ denotes the number of residue classes mod $q$ that are coprime to $q$), and denote by $\mathcal{A}_n(q)$ the set of ordered $n$-tuples $(a_1,a_2,\dots, a_n)$ of distinct residue classes that are coprime to $q$. For $(a_1,a_2,\dots, a_n)\in \mathcal{A}_n(q)$ consider a game with $n$ players called ``$a_1$'' through to ``$a_n$'', where at time $x$, the player $a_j$ has a score of $\pi(x;q,a_j)$ (where $\pi(x;q,a)$ denotes the number of primes $p\leq x$ with $p\equiv a\bmod q$). Among the questions that Knapowski and Tur\'an  asked in \cite{KT} are the following:
\begin{itemize}
\item[Q1.] Will each player take the lead for infinitely many integers $x$?
\item[Q2.] Will all $n!$ orderings of the players occur for infinitely many integers $x$?
\end{itemize}

It is generally believed that the answer to the stronger question Q2 (and thus to Q1) is yes for all $q$ and all $(a_1,a_2,\dots, a_n)\in \mathcal{A}_n(q)$. An old result of Littlewood \cite{Li} shows that this is indeed true when $(q,a_1,a_2)=(4,1,3)$ and $(q,a_1,a_2)=(3,1,2)$. 
Since then, this problem has been extensively studied by various authors, including  Knapowski and Tur\'an \cite{KT},  Kaczorowski \cite{Ka1, Ka2, Ka3}, Feuerverger and Martin \cite{FeM}, Ford and Konyagin \cite{FK1}, Ford, Konyagin and Lamzouri \cite{FKL}, Fiorilli and Martin \cite{FiM}, Fiorilli \cite{Fi}, and Lamzouri \cite{La1, La2}.  For a complete history as well as recent developments, see the expository papers of Granville and Martin \cite{GM},  Ford and Konyagin \cite{FK2}, and Martin and Scarfy \cite{MS} (which includes a very comprehensive list of references). 

Assuming the Generalized Riemann hypothesis GRH and the Linear Independence hypothesis LI (which is the assumption that the nonnegative imaginary parts of the nontrivial zeros of Dirichlet $L$-functions attached to primitive characters are linearly independent over $\mathbb{Q}$), Rubinstein and Sarnak~\cite{RuSa} affirmatively answered questions Q1 and Q2. In fact, under these hypotheses, they established the stronger result that for any $(a_1,\dots, a_n)\in \mathcal{A}_n(q)$, the set of real numbers $x\geq 2$ such that
\begin{equation}\label{FullOrdering}
 \pi(x;q,a_1)>\pi(x;q,a_2)>\dots>\pi(x;q,a_n),
 \end{equation}
  has a positive logarithmic density, which we denote by $\delta(q;a_1, \dots, a_n)$. (Recall that the logarithmic density of a subset $S$ of $\mathbb{R}$ is defined as
$$ \lim_{x\to \infty}\frac{1}{\log x}\int_{t\in S\cap [2,x]}\frac{dt}{t},$$
provided that this limit exists.)  This density can be regarded as the ``probability'' that for each $1\leq j\leq n$, the player $a_j$ is at the $j$-th position in the prime race.

Among their results on question Q2,  Rubinstein and Sarnak \cite{RuSa} showed that for $n$ fixed,
\begin{equation}
\lim_{q\to\infty}\max_{(a_1,\dots,a_n)\in\mathcal{A}_n(q)}\left|\delta(q;a_1,\dots,a_n)-\frac{1}{n!}\right|=0.
\end{equation}
Feuerverger and Martin~\cite{FeM} raised the question of having a uniform version of this statement, in which the number of contestants $n \to \infty$ as $q\to \infty$. In response to this, Lamzouri \cite{La1} established that for any integer $n$ such that $2\leq n\leq \sqrt{\log q}$ we have
 $$\delta(q;a_1, \dots, a_n)= \frac{1}{n!}\left(1+O\left(\frac{n^2}{\log q}\right)\right),$$
uniformly for all $n$-tuples $(a_1,\dots,a_n)\in\mathcal{A}_n(q)$. Feuerverger and Martin~\cite{FeM} also asked whether for $n$ sufficiently large in terms of $q$ the asymptotic formula $\delta(q;a_1, \dots, a_n)\sim 1/n!$ might become false. A few years ago, Ford and Lamzouri (unpublished) developed a heuristic argument which suggests that there should be a transition in the behaviour of the densities when $n=(\log q)^{1+o(1)}$. More specifically, they formulated the following conjecture.
\begin{Con}[Ford and Lamzouri]\label{ExpectedRange}  Let $\varepsilon>0$ be small and $q$ be sufficiently large. 
\begin{enumerate}
\item If $2\leq n\leq (\log q)^{1-\varepsilon}$, then uniformly for all $n$-tuples $(a_1,\dots,a_n)\in\mathcal{A}_n(q)$ we have
$\delta(q;a_1, \dots, a_n)\sim 1/n!$ as $q\to \infty$.

\item If $ (\log q)^{1+\varepsilon} \leq n\leq \varphi(q)$, then there exist  $n$-tuples $(a_1,\dots,a_n), (b_1,\dots, b_n)\in\mathcal{A}_n(q)$ for which $n! \cdot \delta(q;a_1, \dots, a_n) \to 0$ and $n! \cdot \delta(q;b_1, \dots, b_n) \to \infty$ as $q\to\infty$.

\end{enumerate} 
\end{Con} 

Our first result establishes part (1) of this conjecture.
\begin{thm}\label{FullRace}
Assume GRH and LI. Let $2\leq n\leq \log q/(\log\log q)^4$ be a positive integer.  Then, uniformly for all $n$-tuples $(a_1,\dots,a_n)\in\mathcal{A}_n(q)$ we have
$$\delta(q;a_1, \dots, a_n)= \frac{1}{n!}\left(1+O\left(\frac{n(\log n)^4}{\log q}\right)\right).$$
\end{thm}
The second part of Conjecture \ref{ExpectedRange} implies, in particular, that the asymptotic formula $\delta(q;a_1, \dots, a_n)\sim 1/n!$ need not hold for all $n$-tuples $(a_1,\dots,a_n)\in\mathcal{A}_n(q)$, if $n$ lies in the range $(\log q)^{1+\varepsilon} \leq n\leq \varphi(q)$. We believe this to be true because we encounter various error terms in our arguments, from different sources, of size about $1/q$. Thus we believe that when the target density $1/n!$ becomes much smaller than any negative power of $q$, which happens when $n\geq  (\log q)^{1+\varepsilon}$, it is no longer reasonable to expect all of the $\delta(q;a_1, \dots, a_n)$ to be close to $1/n!$. Our next result shows that this is indeed the case in the smaller range $\varphi(q)^{\varepsilon}<n \leq \varphi(q)$. Thus we are able to exhibit, for the first time, prime number races for which the biases do not dissolve when $q\to \infty$, confirming the prediction of Feuerverger and Martin \cite{FeM}.  
\begin{thm}\label{BigBias}
Assume GRH and LI. Let $\varepsilon > 0$ and let $q$ be sufficiently large in terms of $\varepsilon$. For every integer $\varphi(q)^{\varepsilon} \leq n \leq \varphi(q)$ there exists an $n$-tuple $(a_1,\dots,a_n)\in\mathcal{A}_n(q)$ such  that
$$ \delta(q;a_1, \dots, a_n)<\big(1-c_{\varepsilon}\big)\frac{1}{n!},$$
for some positive constant $c_{\varepsilon}$ which depends only on $\varepsilon$. 
\end{thm}

We next consider a stronger form of question Q1, concerning the leader in a prime number race with many contestants.  By the work of Rubinstein and Sarnak, it follows that for any $(a_1,\dots, a_n)\in \mathcal{A}_n(q)$, the set of real numbers $x\geq 2$ such that 
$$
\pi(x;q, a_1)> \max_{2\leq j\leq n}\pi(x;q, a_j),
$$
has a positive logarithmic density which we denote by $\delta_1(q; a_1, \dots, a_n)$. Kaczorowski~\cite{Ka1} has considered this leadership question (in the special case $a_1 =1$), and obtained some positive lower density results assuming only GRH rather than LI. One can ask the following natural quantitative question:
\begin{itemize}
\item[Q3.] Will each of the players $a_1, \dots, a_n$ have an ``equal chance'' $1/n$ of leading the race, when $q\to \infty$?
\end{itemize}
It follows from Theorem \ref{FullRace} that the answer to this question is yes, if the number of contestants $n$ lies in the range $2\leq n = o(\log q/(\log\log q)^4)$. Using a different approach we extend this significantly, showing that $n$ can be as large as a small power of $q$.

\begin{thm}\label{Leader}
Assume GRH and LI. Let $2 \leq n\leq \varphi(q)^{1/32}$ be an integer. Then, uniformly for all $n$-tuples $(a_1,\dots,a_n)\in\mathcal{A}_n(q)$ we have

 $$ \delta_1(q; a_1, \dots, a_n)=\frac{1}{n}\left(1+O\left(\frac{n^4}{\varphi(q)^{1/8}}+ \frac{1}{(n\log q)^{12/25}}\right)\right).$$
\end{thm}
The key ingredient in the proof of this theorem is the following probabilistic result, which may be of independent interest. It investigates the probability that a given random variable is the leader among weakly correlated Gaussian random variables. 
\begin{thm}\label{probleader}
Suppose $\varepsilon > 0$ is sufficiently small, and $n$ is sufficiently large. Let $X_{1},...,X_{n}$ be mean zero, variance one, jointly normal random variables, and write $r_{i,j} := \E X_{i}X_{j}$, and suppose that $|r_{i,j}| \leq \varepsilon$ whenever $i \neq j$. Then
$$ |\p(X_{1} > \max_{2 \leq i \leq n} X_{i}) - 1/n| \ll n^{-100} + n^{-1.99}\sum_{2 \leq i \leq n} |r_{1,i}| + n^{-2.99} \sum_{2 \leq i < j \leq n} |r_{i,j}| . $$
\end{thm}

Note that the probability would be exactly $1/n$ if the $X_i$ were independent of one another, by symmetry.

We also consider a variant of question Q3 concerning the ordering of the first $k$ contestants in a prime race with many competitors. To this end, for each integer $1\leq k\leq n$, we define
 $\delta_k(q;a_1, \dots, a_n)$ to be the logarithmic density of the set of real numbers $x\geq 2$ such that 
$$
\pi(x;q, a_1)>\pi(x;q, a_2) > \cdots> \pi(x;q, a_k)> \max_{k+1\leq j\leq n}\pi(x;q, a_j).
$$
Note that the cases $k=n-1$ and $k=n$ both correspond to the full ordering \eqref{FullOrdering}, and hence $\delta_{n-1}(q;a_1,\dots, a_n)=\delta_{n}(q;a_1,\dots, a_n) = \delta(q;a_1,\dots, a_n)$. 

It follows from Theorem \ref{FullRace} that in the range $1 \leq k \leq n - 1 \leq \log q/(\log\log q)^4$, we have 
$$ \delta_k(q;a_1, \dots, a_n)=\frac{(n-k)!}{n!}\left(1+O\left(\frac{n(\log n)^4}{\log q}\right)\right).$$
Now $\frac{(n-k)!}{n!} = n^{-k(1+o(1))}$ as $n \rightarrow \infty$, so the heuristic discussed following Theorem \ref{FullRace} leads us to expect that $\delta_k(q;a_1, \dots, a_n) \sim \frac{(n-k)!}{n!}$ even for very large $n$, provided roughly that $k \ll (\log q)/\log n$. We will show that the asymptotic does holds on almost this entire range of $k$, for all $n$ such that $(\log n)/\log q\to 0$ as $q\to\infty$. Moreover, unlike the case $k=1$ (where Theorem \ref{Leader} permits $n$ to be a small fixed power of $q$), we show that the condition $(\log n)/\log q\to 0$ is necessary to guarantee an asymptotic for any $k \geq 2$.

\begin{thm}\label{OrderFirstk1}
Assume GRH and LI. Let $2\leq k\leq n-2$ be integers, and suppose $(\log q)/\log n$ is large enough and $k(\log k)^{10} \leq (\log q)/\log n$.  Then, uniformly for all $n$-tuples $(a_1,\dots,a_n)\in\mathcal{A}_n(q)$ we have
$$ \delta_k(q;a_1, \dots, a_n)= \frac{(n-k)!}{n!}\left(1+O\left(k(\log k)^6\frac{\log n}{\log q}+\frac{1}{n \log^{1/10}q} \right)\right).$$
\end{thm}

\begin{thm}\label{OrderFirstk2}
Assume GRH and LI. Let $\varepsilon > 0$ and let $q$ be sufficiently large in terms of $\varepsilon$. Let $k\geq 2$ be fixed and let $n$ be an integer in the range $\varphi(q)^{\varepsilon} \leq n<\varphi(q)^{1/41}$. Then there exists an $n$-tuple $(a_1,\dots,a_n)\in\mathcal{A}_n(q)$ such that
$$ \delta_k(q;a_1, \dots, a_n)<\big(1-c_{\varepsilon}\big)\frac{(n-k)!}{n!},$$
for some positive constant $c_{\varepsilon}$ that depends only on $\varepsilon$. 
\end{thm}
Theorem \ref{BigBias} can be deduced from Theorem \ref{OrderFirstk2} as follows. Take $k=2$, let $\varepsilon>0$ be suitably small, and suppose first that $\varphi(q)^{\varepsilon} \leq n <\varphi(q)^{1/41}$ is an integer. Then it follows from Theorem \ref{OrderFirstk2} that there exists an $n$-tuple $(a_1,\dots,a_n)\in\mathcal{A}_n(q)$ such that 
\begin{equation}\label{BiasedDensity2}
 \delta_2(q; a_1, \dots, a_n)<\big(1-c_{\varepsilon}\big)\frac{(n-2)!}{n!}
 \end{equation} for some positive constant $c_{\varepsilon}$. Let $S_{n-2}$ be the symmetric group on $n-2$ elements. Since the logarithmic density of the set of real numbers $x\geq 2$ for which $\pi(x;q,a)=\pi(x;q,b)$ is $0$ (which follows from equation \eqref{DensityMeasure} below), we get
$$ \delta_2(q; a_1, \dots, a_n)= \sum_{\sigma\in S_{n-2}}\delta(q; a_1, a_2, a_{\sigma(3)}, \dots, a_{\sigma(n)}).$$
Thus by \eqref{BiasedDensity2}, there exists $\sigma\in S_{n-2}$ for which 
$$ \delta(q; a_1, a_2, a_{\sigma(3)}, \dots, a_{\sigma(n)})< \big(1-c_{\varepsilon}\big)\frac{1}{n!},$$
completing the proof of Theorem \ref{BigBias} provided $n < \varphi(q)^{1/41}$.

However, if $n$ is larger then we can set $m := \lfloor \varphi(q)^{1/41} \rfloor - 1$, and note that by the previous discussion there exists an $m$-tuple $(a_1,\dots,a_m)\in\mathcal{A}_m(q)$ for which
$$ \delta(q;a_1, \dots, a_m)<\big(1-c \big)\frac{1}{m!} . $$
Then if we choose any other coprime residues $a_{m+1}, ..., a_n$ mod $q$, we have
$$ \delta(q;a_1, \dots, a_m) = \sum_{\substack{\sigma \in S_{n} : \\ \sigma^{-1}(1) > \sigma^{-1}(2) > ... > \sigma^{-1}(m)}} \delta(q; a_{\sigma(1)}, a_{\sigma(2)}, \dots, a_{\sigma(n)}) . $$
There are $n!/m!$ terms in the sum, so it follows that for at least one permutation $\sigma$ we must have $\delta(q; a_{\sigma(1)}, a_{\sigma(2)}, \dots, a_{\sigma(n)}) < (1-c)\frac{1}{n!}$, as claimed.
\qed

\vspace{12pt}
Next we shall try to indicate the main ideas in the proofs of our theorems.

The work of Rubinstein and Sarnak~\cite{RuSa} showing the existence of the logarithmic densities $\delta(q;a_1, \dots, a_n)$ (assuming GRH and LI) in fact shows that
\begin{equation}\label{RubinSarnak}
 \delta(q;a_1, \dots, a_n) = \p(X(q,a_1) > X(q,a_2) > ... > X(q,a_n)) , 
 \end{equation}
where each $X(q,a)$ is a sum of the same independent random variables twisted by certain arithmetic coefficients depending on $q$ and $a$. Using a quantitative multivariate form of the central limit theorem (which we extract from Stein's method, replacing direct and somewhat messy characteristic function calculations in the previous literature), one can replace the $X(q,a)$ by jointly Gaussian random variables with the same means and covariances. Since the behaviour of Gaussians is entirely determined by those means and covariances, our task is then to obtain as much information as possible about them (on the number theory side), and deduce the best results we can on the ordering probabilities (on the probabilistic side).

Theorem \ref{FullRace} uses a relatively naive probabilistic treatment, namely a direct estimation of the relevant multivariate Gaussian density. The improvement over the result of Lamzouri~\cite{La1} comes from substantially improved estimates for the average size of the covariances feeding into that density (they are typically small, so we are close to a standard multivariate Gaussian). These estimates rely on a harmonic analysis lemma related to the Hardy--Littlewood method, which we use to deduce that the differences $a_i - a_j$ cannot too often be divisible by large divisors of $q$.

All of our other theorems exploit more sophisticated tools for comparing a multivariate Gaussian distribution with the standard multivariate Gaussian, such as the famous comparison lemmas of Slepian (see e.g. Piterbarg~\cite{pit}) and Li and Shao~\cite{lishao}. However, none of these tools seem directly able to prove our theorems, because the probabilities of the events we are interested in are rather small and the bounds we have on the off-diagonal covariances are comparatively large. For example, in the case of Theorem \ref{Leader} we need to show that
$$ \delta_1(q; a_1, \dots, a_n) = \p(X(q,a_1) > \max_{2 \leq j \leq n} X(q,a_j)) = \frac{1}{n}(1+o(1)) , $$
where potentially $1/n$ is as small as $1/\varphi(q)^{1/32} = 1/q^{1/32 + o(1)}$, but where the largest off-diagonal covariances of the $X(q,a)$ (when they are normalised to have variance 1) are $\asymp 1/\log q$. To address this problem, we observe that if $Z_1, Z_2 , ..., Z_n$ are independent standard normal random variables, then with very high probability we have
$$ \max_{2 \leq j \leq n} Z_j = \sqrt{(2-o(1))\log n} , \;\;\;\;\; \text{and} \;\;\;\;\; \p(Z_1 > \sqrt{(2-o(1))\log n}) = \frac{1}{n^{1+o(1)}} . $$
In other words, {\em most of the size $1/n$ of $\p( Z_1 > \max_{2 \leq j \leq n} Z_j)$ is determined just by the probability of $Z_1$ being large enough to possibly be the leader}. This means that we can ``factor out'' most of the small size of our target probability $1/n$ by first conditioning on $X(q,a_1)$ being roughly large enough to be in the lead, leaving a more achievable error bound to be obtained from comparison inequalities. Since the $X(q,a_i)$ are not really independent of one another, this conditioning step itself requires some work and the use of Slepian's Lemma. Theorem \ref{probleader} is a general probabilistic statement that we will prove using these methods.

Our result on the first $k$ places in the race, Theorem \ref{OrderFirstk1}, is proved by combining the direct density arguments of Theorem \ref{FullRace} with the conditioning arguments of Theorem \ref{Leader}. However, since the size of our target probability is now much smaller compared with $n$ the latter part of the argument becomes more challenging, and in fact will be the hardest element of this paper. In particular, we need to prove a modified normal comparison lemma incorporating within it an application of Slepian's lemma, and when applying this we exploit the remarkable known fact that if the correlations of the random variables $X(q,a)$ have positive sign, then they are very small.

Finally, the negative result Theorem \ref{OrderFirstk2} is proved by noting that in an event
$$ X(q,a_1) > X(q,a_2) > ... > X(q,a_k) > \max_{k+1 \leq j \leq n} X(q,a_j) , $$
with very high probability the random variables $X(q,a_1), ..., X(q,a_k)$ will have (normalised) size $\sqrt{(2-o(1))\log n}$, and so in the relevant probability integral there will be terms in the exponential of size $\asymp \log n$. In particular, if there is a correlation between the random variables $X(q,a_i)$ of size about $1/\log q$, this will appear in the density in the exponential and will noticeably distort the multivariate normal probability if $(\log n)/\log q$ isn't small. So by choosing the tuple $(a_1,...,a_n)$ so that there is a correlation of size $\asymp 1/\log q$, which is (the maximum) possible, we obtain an ordering probability that does not converge to uniformity.

\vspace{12pt}
We end by explaining the organisation of the rest of the paper. In section \ref{pnrsetup} we explicitly state the correspondence between logarithmic densities in prime races and orderings of suitable random variables. In section \ref{covarsec} we prove average estimates for the covariances of those random variables. Section \ref{probtoolssec} contains various probabilistic tools tailored to our needs, as well as the proof of Theorem \ref{probleader} and the new probabilistic preparations for Theorem \ref{OrderFirstk1}. Sections \ref{secFullRace} and \ref{secLeader} are relatively short, and contain the deductions of Theorems \ref{FullRace} and \ref{Leader}. Finally, section \ref{secorderfirstk} contains the somewhat difficult proof of Theorem \ref{OrderFirstk1}, and section \ref{secdeviation} has the proof of our negative result, Theorem \ref{OrderFirstk2}.

We have tried to use notation that will not cause confusion between readers from a more number theoretic or a more probabilistic background, but two brief remarks might be in order. Firstly, we use Vinogradov's notation $\ll$, which has the same meaning as the ``big Oh'' notation (thus $x \ll x/10$, for example). In particular, $\ll$ does {\em not} mean ``much less than''. Secondly, if the implicit constant in a statement depends on some ambient parameter, we may adorn the notation with that parameter to reflect the dependence (e.g. we might write $f(x) = O_{\epsilon}(x^{\epsilon})$, meaning that $|f(x)| \leq C(\epsilon) x^{\epsilon}$ for some $C(\epsilon)$).

\section{Logarithmic densities of prime races and corresponding random variables}\label{pnrsetup}

Let $a_1,\dots, a_n$ be distinct reduced residues modulo $q$, and define
$$ E_{q;a_1,\dots,a_n}(x):=\Big(E(x;q,a_1), \dots, E(x;q,a_n)\Big),$$
where
$$ E(x;q,a):=\frac{\log x}{\sqrt{x}}\left(\varphi(q)\pi(x;q,a)-\pi(x)\right) , $$
and $\pi(x)$ denotes the total number of primes less than $x$. It turns out that the normalization is such that, if we assume GRH, $E_{q;a_1,\dots,a_n}(x)$ varies roughly boundedly as $x$ varies. Notice also that
$$ \pi(x;q,a_1) > \pi(x;q,a_2) > ... > \pi(x;q,a_n) \;\; \iff \;\; E(x;q,a_1) > E(x;q,a_2) > ... > E(x;q,a_n) . $$

For a nontrivial Dirichlet character $\chi$ modulo $q$, we denote by $\{\gamma_{\chi}\}$ the sequence of imaginary parts of the nontrivial zeros of  $L(s,\chi)$. If we assume LI then all of the non-negative values of $\gamma_{\chi}$ are linearly independent over $\Q$, and in particular are distinct. Let $\chi_0$ denote the principal character modulo $q$ and define $\Gamma= \bigcup_{\chi\neq \chi_0\bmod q}\{\gamma_{\chi}\}$. Furthermore, let  $\{U(\gamma_{\chi})\}_{\gamma_{\chi}\in \Gamma, \gamma_\chi > 0}$ be a sequence of independent random variables uniformly distributed on the unit circle.
The work of Rubinstein and Sarnak \cite{RuSa} implies, under GRH and LI, that for any Lebesgue measurable set $S\subset \mathbb{R}^n$ whose boundary has measure zero, the logarithmic density 
$$ \lim_{X\to\infty} \int_{\substack{x\in [2,X] \\ E_{q;a_1,\dots,a_n}(x)\in S}} \frac{dx}{x} =: \delta_{q;a_1,\dots,a_n}(S) $$
exists. Moreover, it follows from their work that
\begin{equation}\label{DensityMeasure}
\delta_{q;a_1,\dots,a_n}(S)= \int_{S} d\mu_{q;a_1,\dots,a_n},
\end{equation} where $\mu_{q;a_1,\dots,a_n}$ is the probability measure corresponding to the random vector 

\noindent $\big(X(q,a_1),\dots,X(q,a_n)\big)$,
where
$$ X(q,a):= -C_q(a)+ \sum_{\substack{\chi\neq \chi_0\\ \chi\pmod q}}\re \left(2\chi(a)\sum_{\gamma_{\chi}>0} \frac{U(\gamma_{\chi})}{\sqrt{\frac14+\gamma_{\chi}^2}}\right),$$
with $C_q(a):=-1+ |\{b \pmod q: b^2\equiv a \pmod q\}|$. Note that for $(a,q)=1$ the function $C_q(a)$ takes only two values: $C_q(a)=-1$ if $a$ is a non-square modulo $q$, and $C_q(a)=C_q(1)$ if $a$ is a square modulo $q$. An elementary argument shows that $C_q(a)<d(q)\ll_{\epsilon}q^{\epsilon}$ for any $\epsilon>0$, where $d(q)=\sum_{m|q}1$ is the usual divisor function. Thus it will turn out that the shifts $C_q(a)$ can essentially be ignored when $q \rightarrow \infty$.

Let $\text{Cov}_{q;a_1,\dots,a_n}$ be the covariance matrix of $\big(X(q,a_1),\dots,X(q,a_n)\big)$. Then a straightforward computation (see also Lemma 2.1 of \cite{La1}, for example) shows that 
$$\textup{Cov}_{q;a_1,\dots,a_n}(i,j)= \begin{cases}  \var(q) &\text{ if } i=j \\ B_q(a_i,a_j) &\text{ if } i\neq j,\end{cases}$$
where
$$
\var(q):=2\sum_{\substack{\chi\neq \chi_0\\ \chi\pmod q}}\sum_{\gamma_{\chi}>0}\frac{1}{\frac14+\gamma_{\chi}^2}, \text { and }
B_q(a,b):=\sum_{\substack{\chi\neq \chi_0 \\ \chi\pmod q}}\sum_{\gamma_{\chi}>0}\frac{\chi\left(\frac{b}{a}\right)+\chi\left(\frac{a}{b}\right)}{\frac14 +\gamma_{\chi}^2}.
$$
We end this section by recording several basic estimates for the quantities $\var(q)$ and $B_q(a, b)$ that will be useful in our subsequent work.
\begin{lem}\label{VarianceZeroes}
Assume GRH. Then for any non-principal character $\chi \pmod q$,
\begin{equation}\label{BoundSumZeros}
\sum_{\gamma_{\chi}>0}\frac{1}{\frac14+\gamma_{\chi}^2}\ll \log q.
\end{equation}
Moreover, we have 
\begin{equation}\label{AsympVariance}
\var(q)\sim \varphi(q)\log q \;\;\;\;\; \text{as} \; q \rightarrow \infty .
\end{equation} 
\end{lem}
\begin{proof}
These estimates follow from Lemma 3.1 of \cite{La2}, for example.
\end{proof}

\begin{lem}\label{CovarianceEntries}
Assume GRH. For all distinct reduced residues $a, b$ mod $q$, we have 
\begin{equation}\label{BoundCov}
B_q(a.b)\ll \varphi(q).
\end{equation}

Moreover, if $a, b$ are distinct residue classes such that $1\leq |a|<|b|\leq q/2$ and $|b|/|a|$ is not a prime power, then we have 
\begin{equation}\label{SmallCov}
B_q(a, b)\ll |b|(\log q)^2.
\end{equation}
On the other hand we have 
\begin{equation}\label{LargeCov}
B_q(1, -1)= -(\log 2)\varphi(q)+O\left((\log q)^2\right).
\end{equation}

Finally, if $a,b$ are distinct residue classes and $B_{q}(a,b) \geq 0$ then
\begin{equation}\label{PosImpliesSmall}
B_{q}(a,b) \ll \log q .
\end{equation}
\end{lem}

\begin{proof}
The first bound \eqref{BoundCov} corresponds to Corollary 5.4 of \cite{La2}. The estimates \eqref{SmallCov} and \eqref{LargeCov} follow from Proposition 6.1 of \cite{La2}. The fact \eqref{PosImpliesSmall} that positive correlations are always very small is noted in Remark 5.1 of \cite{La2}, for example.
\end{proof}

\section{An average result for the sums of the covariances $B_q(a_i, a_j)$}\label{covarsec}

\subsection{A double average}
In view of Lemma \ref{CovarianceEntries}, all of the non-diagonal covariances in our prime number race satisfy (when normalised by the variance $\var(q)$)
$$ \frac{|B_q(a_i,a_j)|}{\var(q)} \ll \frac{1}{\log q} , \;\;\;\;\; i \neq j . $$
This bound is useful, but to obtain strong results we will need to exploit the fact that, if we are looking at many residue classes $a_1 , ..., a_n$, the covariances will on average be much smaller. This was established in Theorem 5 of Lamzouri~\cite{La2} when averaging over {\em all} pairs of distinct reduced residue classes, but we will need a strong result when averaging only over a subset, which requires quite different methods.

\begin{correst1}
Assume GRH. Let $q$ be large and let $r,s \geq 1$. For any collections $a_{1},...,a_{r}$ and $b_{1},...,b_{s}$ of distinct reduced residue classes modulo $q$, we have
$$ \sum_{1 \leq j \leq r} \sum_{\substack{1 \leq k \leq s, \\ b_{k} \neq a_{j}}} \frac{|B_{q}(a_{j},b_{k})|}{\var(q)} \ll \frac{\sqrt{rs} \log^{2}(2rs)}{\log q} . $$

In particular, we have
$$ \sum_{\substack{1 \leq k \leq s, \\ b_{k} \neq a_{1}}} \frac{|B_{q}(a_{1},b_{k})|}{\var(q)} \ll \frac{\sqrt{s} \log^{2}(2s)}{\log q} , \;\;\; \text{and} \;\;\; \sum_{\substack{1 \leq j,k \leq r, \\ j \neq k}} \frac{|B_{q}(a_{j},a_{k})|}{\var(q)} \ll \frac{r \log^{2}(2r)}{\log q} . $$
\end{correst1}

Note that this estimate saves roughly a factor of $\sqrt{rs}$ as compared with a trivial treatment using the pointwise bound $\frac{|B_{q}(a_{j},b_{k})|}{\var(q)} \ll 1/\log q$.

\vspace{12pt}
To prove Correlation Estimate 1 we shall need two lemmas. The first is the following, which will reduce the problem to upper bounding some easier sums.
\begin{lem}\label{simplercorrupper}
Assume GRH. In the setting of Correlation Estimate 1, for any $1 \leq j \leq r$ we have
$$ \sum_{\substack{1 \leq k \leq s, \\ b_{k} \neq a_{j}}} \frac{|B_{q}(a_{j},b_{k})|}{\varphi(q)} \ll \log^{2}(2s) + \sum_{\substack{1 \leq k \leq s, \\ b_{k} \neq a_{j}}} \frac{\Lambda(q/(q,a_{j}-b_{k}))}{\varphi(q/(q,a_{j}-b_{k}))} , $$
where $\Lambda(n)$ denotes the von Mangoldt function.
\end{lem}

\begin{proof}[Proof of Lemma \ref{simplercorrupper}]
Let $x=(q\log q)^2$, and for simplicity of writing set $a = a_j$. Then it follows from Proposition 5.1 of \cite{La1} that 
\begin{equation}\label{explicit}
\begin{aligned}
 \sum_{\substack{1 \leq k \leq s, \\ b_k \neq a}} |B_q(a,b_k)| \ll \varphi(q)\Bigg(&1+\sum_{\substack{1 \leq k \leq s, \\ b_{k} \neq a}} \Bigg(\frac{\Lambda\left(\frac{q}{(q,a-b_k)}\right)}{\varphi\left(\frac{q}{(q,a-b_k)}\right)}+ M_1(q;a,b_k)+M_1(q;b_k,a)\\
 &+  M_2(q;a,b_k)+M_2(q;b_k,a)\Bigg)\Bigg)+s\log q ,
 \end{aligned}
\end{equation}
where 
$$ M_1(q;a,d)= \sum_{\substack{n\leq 2x\log x\\ an\equiv d\pmod q}}\frac{\Lambda(n)}{n}e^{-n/x}, \text{ and } 
M_2(q;a,d)= \sum_{p^{\nu}\parallel q}\sum_{\substack{1\leq e\leq 2\log x\\ap^e\equiv d  \bmod q/p^{\nu}}}\frac{\log p}{p^{e+\nu-1}(p-1)} , $$
and $p^{\nu}\parallel q$ denotes that $p^{\nu}$ is the largest power of $p$ that divides $q$.

Now it follows from Lemma 5.3 of \cite{La1} that
$$ M_1(q;a,b_k)= \frac{\Lambda(n_k)}{n_k}+O\left(\frac{\log^2q}{q}\right) \leq \frac{\log(n_k)}{n_k}+O\left(\frac{\log^2q}{q}\right) , $$
where $n_k$ is the least positive residue of $b_k a^{-1}$ modulo $q$. Since $(\log n)/n$ is a decreasing function for $n\geq 3$, we deduce that
$$ \sum_{\substack{1 \leq k \leq s, \\ b_{k} \neq a}} M_1(q;a,b_k) \ll \sum_{\substack{1 \leq k \leq s, \\ b_{k} \neq a}} \frac{\log(n_k)}{n_k}+ \frac{s\log^2q}{q} \ll 1 + \sum_{k \leq s} \frac{\log k}{k} + \frac{s\log^2q}{q} \ll \log^{2}(2s) + \frac{s\log^2 q}{q} . $$
A similar bound holds for $\sum M_1(q;b_k,a)$. 

Next we bound the sum $\sum (M_2(q;a,b_k) + M_2(q;b_k,a))$. We have
$$ \sum_{k=1}^s M_2(q;a,b_k) = \sum_{p^{\nu}\parallel q}\sum_{e\leq 2\log x}\frac{\log p}{p^{e+\nu-1}(p-1)}\sum_{\substack{1\leq k \leq s\\ap^{e}\equiv b_k  \bmod q/p^{\nu}}} 1 \leq \sum_{p^{\nu}\parallel q}\sum_{e\leq 2\log x}\frac{\log p}{p^{e+\nu-1}(p-1)} \min\{p^{\nu}, s\} , $$
since the $b_k$ are distinct modulo $q$. Splitting the outer sum over the primes $p$ dividing $q$ into the cases $p\leq s$ and $p>s$, we find the above is
$$ \ll \sum_{p\leq s}\sum_{e=1}^{\infty}\frac{\log p}{p^{e-1}(p-1)} + s \sum_{p>s} \sum_{e=1}^{\infty}\frac{\log p}{p^e(p-1)} \ll \sum_{p\leq s}\frac{\log p}{p} + s \sum_{p>s}\frac{\log p}{p^2} \ll \log(2s) . $$
A similar bound holds for $\sum_{k=1}^s M_2(q;b_k,a)$. Putting everything together (and remembering that $s \leq q$, so $(s \log^{2}q)/q \ll \log^{2}(2s)$) completes the proof of Lemma \ref{simplercorrupper}.
\end{proof}

To control the sum on the right hand side in Lemma \ref{simplercorrupper} we shall deploy the following harmonic analysis lemma. This will be the really new aspect of our analysis of the covariances.
\begin{lem}[Following pp 305-307 of Bourgain~\cite{bourgain}, 1989]\label{harmanal}
Let $x$ be large and let $Q \geq 1$. Define
$$ G(\theta) := \sum_{q \leq Q} \frac{\Lambda(q)}{q} \sum_{a=0}^{q-1} \textbf{1}_{||\theta - a/q|| \leq 1/x} , $$
where $||\cdot||$ denotes distance to the nearest integer, and $\textbf{1}$ denotes the indicator function.

Then if $\theta_{1},...,\theta_{R}$ and $\phi_{1},...,\phi_{S}$ are any real numbers that are $1/x$-spaced (i.e. such that $||\theta_{r_1}-\theta_{r_2}|| \geq 1/x$ when $r_1 \neq r_2$, and such that $||\phi_{s_1}-\phi_{s_2}|| \geq 1/x$ when $s_1 \neq s_2$), we have
$$ \sum_{\substack{1 \leq r \leq R, \\ 1 \leq s \leq S}} G(\theta_{r}-\phi_{s}) \ll \sqrt{RS} \log^{2}(2QRS) + \frac{RSQ}{x} . $$
\end{lem}

Lemma \ref{harmanal} encodes the fact that rationals $a/q$ are well spaced, so it is impossible for lots of the points $\theta_{r}-\phi_{s}$ to be very close to lots of rationals. The proof uses additive characters, the problem being analytically nice because we are looking at pairwise differences, which corresponds to a convolution on the harmonic analysis side. Since Bourgain's argument is given in a very different context, and since our statement of Lemma \ref{harmanal} is also different (in particular through the presence of the weights $\Lambda(q)$), we provide a sketch proof of the lemma in Appendix \ref{analysisappendix}.

\begin{proof}[Proof of Correlation Estimate 1]
We may suppose without loss of generality that $s \geq r$. In view of Lemma \ref{simplercorrupper}, and the fact that $\var(q)\sim \varphi(q)\log q$ that is contained in Lemma \ref{VarianceZeroes}, to prove Correlation Estimate 1 it will certainly suffice to prove that
$$ \sum_{1 \leq j \leq r} \sum_{\substack{1 \leq k \leq s, \\ b_{k} \neq a_{j}}} \frac{\Lambda(q/(q,a_{j}-b_{k}))}{\varphi(q/(q,a_{j}-b_{k}))} \ll \sqrt{rs} \log^{2}(2rs) . $$

Rewriting a little, the double sum is
$$ \leq \sum_{n|q} \frac{\Lambda(n)}{\varphi(n)} \sum_{1 \leq j \leq r} \sum_{\substack{1 \leq k \leq s, \\ b_{k} \neq a_{j}}} \textbf{1}_{a_{j} \equiv b_{k} \; \text{mod} \; q/n} \ll \sum_{n|q} \frac{\Lambda(n)}{n} \sum_{1 \leq j \leq r} \sum_{\substack{1 \leq k \leq s, \\ b_{k} \neq a_{j}}} \textbf{1}_{a_{j} \equiv b_{k} \; \text{mod} \; q/n} , $$
where $\textbf{1}$ denotes the indicator function, and where we used the fact that $\varphi(n) \asymp n$ if $n$ is a prime power (and so $\Lambda(n) \neq 0$). We also observe that we cannot have $a_{j} \equiv b_{k}$ modulo $q/n$ for two values $n$ that are powers of different primes, since in that case we would have $a_{j} \equiv b_{k}$ modulo $q$, which is false by assumption. Therefore the contribution to the sum from those $n > rs$ is trivially
$$ \ll \sum_{1 \leq j \leq r} \sum_{\substack{1 \leq k \leq s, \\ b_{k} \neq a_{j}}} \frac{\log(2rs)}{rs} \leq \log(2rs) , $$
which is acceptable.

To bound the contribution to the sum from those $n \leq rs$, we apply Lemma \ref{harmanal} with the choices $Q=rs$, $x=\max\{q,(rs)^{2}\}$, $\theta_{j} := a_{j}/q$ and $\phi_{k} := b_{k}/q$ for all $j,k$. Thus all the points $\theta_{j}$ and $\phi_{k}$ are indeed $1/x$-spaced, and $a_{j} \equiv b_{k}$ modulo $q/n$ if and only if $\theta_{j}-\phi_{k} = u/n$ for some integer $u$. We deduce that
\begin{eqnarray}
\sum_{n \leq rs, n|q} \frac{\Lambda(n)}{n} \sum_{1 \leq j \leq r} \sum_{\substack{1 \leq k \leq s, \\ b_{k} \neq a_{j}}} \textbf{1}_{a_{j} \equiv b_{k} \; \text{mod} \; q/n} & \leq & \sum_{\substack{1 \leq j \leq r, \\ 1 \leq k \leq s}} \sum_{n \leq rs} \frac{\Lambda(n)}{n} \sum_{u=0}^{n-1} \textbf{1}_{||(\theta_{j}-\phi_{k}) - u/n|| \leq 1/x} \nonumber \\
& \ll & \sqrt{rs} \log^{2}(2rs) + 1 , \nonumber
\end{eqnarray}
which is enough to prove Correlation Estimate 1.
\end{proof}

\subsection{Some matrix estimates involving the covariances}
We will need some information about the determinant and inverse of a covariance matrix that is close to the identity matrix. Let $\mathcal{M}_{n}(\epsilon)$ denote the set of all $n \times n$ symmetric matrices whose diagonal entries are 1, and whose off-diagonal entries have absolute value at most $\epsilon$.
\begin{lem}\label{matrixlem}
If $\epsilon \leq 1/2n$ then for any $A = (a_{j,k}) \in \mathcal{M}_{n}(\epsilon)$ we have
$$ \det(A) = 1 + O\left(\epsilon \sum_{\substack{1 \leq j,k \leq n, \\ j \neq k}} |a_{j,k}| \right). $$

In addition, if $\epsilon \leq 1/2n$ then $A$ is invertible, and if we let $\tilde{a}_{j,k}$ denote the entries of the inverse matrix $A^{-1}$ then we have
$$ \tilde{a}_{j,k} = \left\{ \begin{array}{ll}
     1 + O\left(\epsilon \sum_{\substack{1 \leq l,m \leq n, \\ l \neq m}} |a_{l,m}| \right) & \text{if} \; j=k, \\
     O\left(|a_{j,k}| + \sum_{i \neq j,k} |a_{j,i}||a_{i,k}| + \epsilon^{2} \sum_{\substack{1 \leq l,m \leq n, \\ l \neq m}} |a_{l,m}| \right) & \text{if} \; j \neq k .
\end{array} \right. $$
\end{lem}

Lemma \ref{matrixlem} extends Lemmas 4.1 and 4.2 of Lamzouri~\cite{La1}, which would give roughly the same result if all the off-diagonal entries $a_{j,k}$ had size about $\epsilon$, but are weaker if the $a_{j,k}$ are on average smaller (as will later be the case for us when we take $\epsilon$ of size about $1/\log q$).

\begin{proof}[Proof of Lemma \ref{matrixlem}]
We have
$$ \det(A) = 1 + \sum_{\substack{\sigma \in S_{n}, \\ \sigma \neq 1}} \text{sgn}(\sigma) a_{1,\sigma(1)}...a_{n,\sigma(n)} , $$
where $S_{n}$ denotes the symmetric group on $n$ elements. We divide the sum according to the number $t$ of points that are {\em not} fixed by a permutation $\sigma$. Thus the only term with $t=0$ is the identity permutation, which we removed from the sum; there are no terms with $t=1$; and the contribution from $t=2$ has size at most
$$ \sum_{1 \leq j \leq n} \sum_{\substack{1 \leq k \leq n, \\ k \neq j}} |a_{j,k}|^{2} \leq \epsilon \sum_{\substack{1 \leq j,k \leq n, \\ j \neq k}} |a_{j,k}| . $$
For any $3 \leq t \leq n$, by averaging the total contribution is at most
$$ \frac{1}{t} \sum_{1 \leq j \leq n} \sum_{\substack{\sigma \in S_{n}, \sigma(j) \neq j, \\ \sigma \; \text{has} \; t \; \text{non-fixed points}}} |a_{j,\sigma(j)}| \epsilon^{t-1} = \frac{1}{t} \sum_{\substack{1 \leq j,k \leq n, \\ j \neq k}} |a_{j,k}| \epsilon^{t-1} \sum_{\substack{\sigma \in S_{n}, \sigma(j) = k, \\ \sigma \; \text{has} \; t \; \text{non-fixed points}}} 1 . $$
In the inner sum we have ${n-2 \choose t-2}$ choices of points that are not fixed (in addition to $j$ and $k$), and for any such choice there are at most $(t-1)!$ ways to construct $\sigma$ such that $\sigma(j)=k$. Thus the total contribution from any $3 \leq t \leq n$ is
$$ \leq \sum_{\substack{1 \leq j,k \leq n, \\ j \neq k}} |a_{j,k}| \epsilon^{t-1} \frac{1}{t} {n-2 \choose t-2} (t-1)! \leq \sum_{\substack{1 \leq j,k \leq n, \\ j \neq k}} |a_{j,k}| \epsilon (\epsilon n)^{t-2} , $$
and our claim about $\det(A)$ follows on using the assumption $\epsilon \leq 1/2n$ and summing over $t$. In particular, we see that $\det(A) \geq 1/2$ for all $A \in \mathcal{M}_{n}(\epsilon)$, so $A$ is invertible.

Next we need to prove the claims about $\tilde{a}_{j,k}$. If $j=k$ then we have
$$ \tilde{a}_{j,j} = \frac{\det(A_{j,j})}{\det(A)} , $$
where $A_{j,j}$ denotes the matrix $A$ with the $j$-th row and column removed. In particular we have $A_{j,j} \in \mathcal{M}_{n-1}(\epsilon)$, and so the claim about $\tilde{a}_{j,j}$ follows from our determinant results.

In the off-diagonal case we have
$$ |\tilde{a}_{j,k}| = \left|\frac{\det(A_{k,j})}{\det(A)} \right| = O(|\det(A_{k,j})|) = O\left(\sum_{\sigma \in \mathcal{B}_{k,j}} \prod_{i \neq k} |a_{i,\sigma(i)}| \right) , $$
where now $\mathcal{B}_{k,j}$ denotes the set of all bijections from $\{1,2,...,n\} \backslash \{k\}$ to $\{1,2,...,n\} \backslash \{j\}$. We again divide the sum according to the number $t$ of points that are {\em not} fixed by $\sigma$, noting that since $k \neq j$ the point $j$ is necessarily always a non-fixed point, and the point $k$ will always be the image of a non-fixed point. Thus there are no terms with $t=0$, and the contribution from $t=1$ is simply $|a_{j,k}|$. When $t=2$ the contribution is
$$ \sum_{i \neq j,k} |a_{j,i}||a_{i,k}| . $$
For any $3 \leq t \leq n-1$, each $\sigma$ must have at least $t-2$ non-fixed points that are different from $j$, and whose image is different from $k$, so by averaging the total contribution is at most
$$ \frac{1}{t-2} \sum_{\substack{1 \leq i \leq n, \\ i \neq k, j}} \sum_{\substack{\sigma \in \mathcal{B}_{k,j}, \sigma(i) \neq i,j,k, \\ \sigma \; \text{has} \; t \; \text{non-fixed points}}} |a_{i,\sigma(i)}| \epsilon^{t-1} = \frac{1}{t-2} \sum_{i \neq k,j} \sum_{h \neq i,j,k} |a_{i,h}| \epsilon^{t-1} \sum_{\substack{\sigma \in \mathcal{B}_{k,j}, \sigma(i) = h, \\ \sigma \; \text{has} \; t \; \text{non-fixed points}}} 1 . $$
In the inner sum we have ${n-4 \choose t-3}$ choices of points that are not fixed (in addition to $j$, $i$ and $h$), and for any such choice there are at most $(t-1)!$ ways to construct $\sigma$ such that $\sigma(i)=h$. Thus the total contribution from any $3 \leq t \leq n-1$ is
$$ \leq \sum_{\substack{1 \leq i,h \leq n, \\ i \neq h}} |a_{i,h}| \epsilon^{t-1} \frac{1}{t-2} {n-4 \choose t-3} (t-1)! \leq \sum_{\substack{1 \leq i,h \leq n, \\ i \neq h}} |a_{i,h}| \epsilon^{2} (\epsilon n)^{t-3} (t-1) . $$
Proposition 1 follows on using the assumption $\epsilon \leq 1/2n$ and summing over $t$.
\end{proof}

\section{Probabilistic tools and results}\label{probtoolssec}

\subsection{Passing to the Gaussian case}
As described in section \ref{pnrsetup}, the prime number race between residue classes $a$ modulo $q$ is associated with random variables of the general shape
$$ W_{a} := \sum_{i=1}^{m} \Re(c_{i}(a) V_{i}) , $$
where $(V_{i})_{1 \leq i \leq m}$ is a sequence of independent, mean zero, complex valued random variables (depending on $q$), and where $c_{i}(a) \in \C$ are deterministic coefficients.

In order to access the tools associated with Gaussian random processes, we would like a multivariate normal approximation (i.e. multivariate central limit theorem) for the $n$-dimensional random vector $W=(W_{a_j})_{1 \leq j \leq n}$. This will allow us to replace the $W_{a_j}$ by Gaussian random variables with the same means and covariances. We need an explicit bound for the error arising in the normal approximation, which in particular makes clear the dependence on $n$. There are not too many such results in the literature, and we will deduce a suitable result from the work of Reinert and R\"{o}llin~\cite{rr}.
\begin{lem}[Following Theorem 2.1 of Reinert and R\"{o}llin~\cite{rr}, 2009]\label{reinrollnormal}
Let the situation be as described above, let $\mathcal{A}$ be a finite set, and let $Z=(Z_{a})_{a \in \mathcal{A}}$ denote a multivariate normal random vector with the same mean vector and covariance matrix as $W=(W_{a})_{a \in \mathcal{A}}$. Assume that $\E|V_i|^{4} \leq K^{4}/m^{2}$ for all $i$, for some $K \geq 1$.

Then for any three times differentiable function $h : \R^{\#\mathcal{A}} \rightarrow \R$ we have
$$ |\E h(W) - \E h(Z)| \ll \frac{|h|_2 K^{2}}{m} \sum_{a,b \in \mathcal{A}} \sqrt{\sum_{i=1}^{m} |c_{i}(a)|^{2}|c_{i}(b)|^{2}} + \frac{|h|_3 K^{3}}{m^{3/2}} \sum_{i=1}^{m} \left(\sum_{a \in \mathcal{A}} |c_{i}(a)| \right)^{3}, $$
where $|h|_2 := \sup_{a,b \in \mathcal{A}} ||\frac{\partial^2}{\partial x_{a} \partial x_{b}} h||_{\infty}$ and $|h|_3 := \sup_{a,b,c \in \mathcal{A}} ||\frac{\partial^3}{\partial x_{a} \partial x_{b} \partial x_{c}} h||_{\infty}$.
\end{lem}

Reinert and R\"{o}llin's work develops a multivariate version of Stein's method of exchangeable pairs, and applies in far more general situations than above. In Appendix \ref{normapproxappend} we very briefly indicate how to deduce Lemma \ref{reinrollnormal} from Theorem 2.1 of Reinert and R\"{o}llin~\cite{rr} (this deduction being, by now, a fairly standard calculation).

For the special case of prime number races, we centre and normalize the random variables $X(q,a_j)$ from Section \ref{pnrsetup} by setting $$ Y_{j} := \frac{X(q,a_j) + C_{q}(a_{j})}{\sqrt{\var(q)}} = \frac{1}{\sqrt{\var(q)}} \sum_{\substack{\chi \; \text{mod} \; q, \\ \chi \neq \chi_{0}}} \Re\left(2 \chi(a_{j}) \sum_{\gamma_{\chi} > 0} \frac{U(\gamma_{\chi})}{\sqrt{\frac{1}{4} + \gamma_{\chi}^{2}}} \right) , \;\;\;\;\; 1 \leq j \leq n, $$
where the $U(\gamma_{\chi})$ are independent random variables distributed uniformly on the unit circle. Then $Y_1,\dots, Y_n$ have mean zero and variance $1$. Moreover, we have
 $$\ex Y_iY_j=\frac{B_q(a_i,a_j)}{\var(q)}\ll \frac{1}{\log q}, \;\;\;\;\; i \neq j , $$
 by  \eqref{AsympVariance} and \eqref{BoundCov}. 
In this case, we obtain the following corollary.
\begin{lem}\label{pnrfirstapprox}
Let $Z=(Z_{j})_{1 \leq j \leq n}$ denote a multivariate normal random vector whose components have mean zero, variance one, and correlations
$$ \E Z_{i}Z_{j} := \E Y_{i}Y_{j} = \frac{B_{q}(a_i,a_j)}{\var(q)} . $$
Then for any three times differentiable function $h : \R^{n} \rightarrow \R$ we have
$$ |\E h(Y) - \E h(Z)| \ll \frac{n^{2}|h|_2 + n^{3}|h|_3}{\sqrt{\varphi(q)}} . $$
\end{lem}

\begin{proof}[Proof of Lemma \ref{pnrfirstapprox}]
We apply Lemma \ref{reinrollnormal} with the sum over $1 \leq i \leq m$ replaced by a sum over characters $\chi \neq \chi_{0}$ mod $q$ (so $m = \varphi(q) - 1$), and with $c_{i}(a_j)$ replaced by $\chi(a_{j})$ and with $V_{i}$ replaced by
$$ V_{\chi} := \frac{2}{\sqrt{\var(q)}} \sum_{\gamma_{\chi} > 0} \frac{U(\gamma_{\chi})}{\sqrt{\frac{1}{4} + \gamma_{\chi}^{2}}} . $$

The $V_{\chi}$ are indeed independent, mean zero random variables. Moreover, observe that $\E U(\gamma_{\chi_1}) U(\gamma_{\chi_2}) \overline{U(\gamma_{\chi_3}) U(\gamma_{\chi_4})}$ vanishes unless $\{\chi_1 , \chi_2 \} = \{\chi_3 , \chi_4 \}$, and in that case it is $\ll 1$. Therefore we have
$$ \E |V_{\chi}|^{4} \ll \frac{1}{\var(q)^{2}} \left(\sum_{\gamma_{\chi} > 0} \frac{1}{\frac{1}{4} + \gamma_{\chi}^{2}} \right)^{2} \ll \frac{\log^{2}q}{\var(q)^{2}} \ll \frac{1}{\varphi(q)^{2}} , $$
where the final inequalities follow from Lemma \ref{VarianceZeroes}. So we have checked that Lemma \ref{reinrollnormal} is applicable with $K$ an absolute constant.

Using the facts that $|c_{i}(a_j)| = |\chi(a_{j})| \leq 1$ and $\#\mathcal{A} = n$ and $m = \varphi(q) - 1$, Lemma \ref{pnrfirstapprox} now follows immediately from Lemma \ref{reinrollnormal}.
\end{proof}

The above still isn't quite what we need, since we are interested in the probabilities of certain orderings of the $Y_{j}$ (or really of the $X(q,a_j)$), and these correspond to the expectations of indicator functions that are not three times differentiable. So we need to approximate indicator functions by smooth functions $h$, with some control on the resulting error in the probabilities. There is a substantial literature that attempts to do this as efficiently as possible, but for us a simple approach will suffice.
\begin{lem}\label{pnrunsmoothing}
Let $Y=(Y_{j})_{1 \leq j \leq n}$ and $Z=(Z_{j})_{1 \leq j \leq n}$ be as in Lemma \ref{pnrfirstapprox}. Let $(S)$ be any subset of $\{1,...,n\} \times \{1,...,n\}$ not including any diagonal pairs $(i,i)$, and let $S$ be the subset of $\R^{n}$ defined by
$$ S := \{(x_1,\dots,x_n) \in \R^{n} : x_{i} \geq x_{j} \; \forall (i,j) \in (S)\} . $$
Finally, let $\delta > 0$ be a small parameter.

Then
$$ |\p(Y \in S) - \p(Z \in S)| \ll \frac{n^{6}}{\delta^{3} \sqrt{\varphi(q)}} + n^{2} \delta . $$
\end{lem}

\begin{proof}[Proof of Lemma \ref{pnrunsmoothing}]
Let $\phi :\mathbb{R}\to \mathbb{R}$ be a three times differentiable function such that 
$$ \phi(x)= \begin{cases} 1 &\text{ if } x\geq \delta \\ \in [0,1] &\text{ if } 0<x<\delta \\0 &\text{ if } x\leq 0,\end{cases}$$
and  $\phi^{(d)}(x)\ll (1/\delta)^d$ for $1\leq d\leq 3$ (where $\phi^{(d)}$ denotes the $d$-th derivative of $\phi$). Note that such $\phi$ exists since the interval on which $\phi$ changes from $0$ to $1$ has length $\delta$.

Now, let $h^{+}_{\delta}, h^{-}_{\delta}:\mathbb{R}^n\to \mathbb{R}$ be three times differentiable functions defined by
$$ h^{-}_{\delta}(x_1,\dots,x_n) := \prod_{(i,j) \in (S)} \phi(x_i-x_j) , \;\;\;\;\; h^{+}_{\delta}(x_1,\dots,x_n) := \prod_{(i,j) \in (S)} \phi(x_i-x_j + \delta). $$
By repeated application of the product rule we see
\begin{eqnarray}
\frac{\partial^2}{\partial x_a \partial x_b}h^{-}_{\delta}(x_1,\dots,x_n) & = & \sum_{\substack{(i,j), (k,l) \in (S), \\ (i,j) \neq (k,l)}} \Biggl(\prod_{\substack{(u,v) \in (S), \\ (u,v) \neq (i,j), (k,l)}} \phi(x_u-x_v) \Biggr) \frac{\partial}{\partial x_a} \phi(x_{i}-x_{j}) \frac{\partial}{\partial x_b} \phi(x_{k}-x_{l}) \nonumber \\
&& + \sum_{(i,j) \in (S)} \Biggl(\prod_{\substack{(u,v) \in (S), \\ (u,v) \neq (i,j)}} \phi(x_u-x_v) \Biggr) \frac{\partial^2}{\partial x_a \partial x_b} \phi(x_{i}-x_{j}) . \nonumber
\end{eqnarray}
Here each of the products has absolute value at most 1. The derivatives vanish unless $a \in \{i,j\}$ and $b \in \{k,l\}$, so there are $\ll n^{2}$ non-vanishing terms in the sums, and each term contributes $\ll 1/\delta^{2}$ because $\phi^{(d)}(x)\ll (1/\delta)^d$. We conclude from this calculation that $|h^{-}_{\delta}|_2 \ll n^{2}/\delta^{2}$. An exactly similar argument shows that $|h^{-}_{\delta}|_3 \ll n^{3}/\delta^{3}$, and the same for $h^{+}_{\delta}$.

Now observe that $\E h^{-}_{\delta}(Y) \leq \p(Y \in S) \leq \E h^{+}_{\delta}(Y)$. In addition, by Lemma \ref{pnrfirstapprox} we have
$$ |\E h^{-}_{\delta}(Y) - \E h^{-}_{\delta}(Z)| \ll \frac{n^{6}}{\delta^{3} \sqrt{\varphi(q)}} \;\;\; \text{and} \;\;\; |\E h^{+}_{\delta}(Y) - \E h^{+}_{\delta}(Z)| \ll \frac{n^{6}}{\delta^{3} \sqrt{\varphi(q)}} , $$
so to prove Lemma \ref{pnrunsmoothing} it will suffice to show that $|\p(Z \in S) - \E h^{-}_{\delta}(Z)| \ll n^{2} \delta$ and $|\p(Z \in S) - \E h^{+}_{\delta}(Z)| \ll n^{2} \delta$. But this is fairly easy, because we have for example that
$$ |\p(Z \in S) - \E h^{-}_{\delta}(Z)| \leq \p(|Z_i - Z_j| \leq \delta \; \text{for some} \; (i,j) \in (S)) \leq \sum_{(i,j) \in (S)} \p(|Z_i - Z_j| \leq \delta) . $$
And each difference $Z_i - Z_j$ is a normal random variable with mean zero and variance
$$ \E(Z_i - Z_j)^{2} = \E Z_i^{2} + \E Z_j^{2} - 2\E Z_i Z_j = 2 - 2 \frac{B_{q}(a_i,a_j)}{\var(q)} = 2 + O(1/\log q) , $$
in view of  \eqref{AsympVariance} and \eqref{BoundCov} (and our assumption that if $(i,j) \in (S)$ then $i \neq j$, and therefore $a_i \neq a_j$). Such a normal random variable has probability $O(\delta)$ of lying in any $\delta$-ball, and there are at most $n^{2}$ pairs in $(S)$, so we have the required bound $\ll n^{2} \delta$ for the right hand side.
\end{proof}

Finally, by making the optimal choice of $\delta$ in Lemma \ref{pnrunsmoothing}, and correcting for the small shifts $C_{q}(a_j)/\sqrt{\var(q)}$ in the definition of $Y_j$, we obtain the following result that we shall actually use.
\begin{normapprox}
Let $X := \left(\frac{X(q,a_j)}{\sqrt{\var(q)}}\right)_{1 \leq j \leq n}$, and let $Z=(Z_{j})_{1 \leq j \leq n}$ denote a multivariate normal random vector whose components have mean zero, variance one, and correlations $\E Z_{i}Z_{j} := \frac{B_{q}(a_i,a_j)}{\var(q)}$.

Then for any set $S$ as in Lemma \ref{pnrunsmoothing}, we have
$$ |\p(X \in S) - \p(Z \in S)| \ll \frac{n^{3}}{\varphi(q)^{1/8}} . $$
\end{normapprox}

\begin{proof}[Proof of Normal Approximation Result 1]
Choosing $\delta = n/\varphi(q)^{1/8}$ in Lemma \ref{pnrunsmoothing} is optimal and leads to the bound
$$ |\p(Y \in S) - \p(Z \in S)| \ll \frac{n^{3}}{\varphi(q)^{1/8}} . $$

Since we have $X = Y - (\frac{C_{q}(a_j)}{\sqrt{\var(q)}})_{1 \leq j \leq n}$, and as described in section \ref{pnrsetup} we always have $C_{q}(a_j)/\sqrt{\var(q)} \ll_{\epsilon} q^{\epsilon}/\sqrt{\var(q)} \ll 1/\varphi(q)^{1/2-\epsilon}$ (which is much smaller than $\delta = n/\varphi(q)^{1/8}$), then as well as the inequality $\E h^{-}_{\delta}(Y) \leq \p(Y \in S) \leq \E h^{+}_{\delta}(Y)$ in the proof of Lemma \ref{pnrunsmoothing} we will actually have $\E h^{-}_{\delta}(Y) \leq \p(X \in S) \leq \E h^{+}_{\delta}(Y)$, provided the function $\phi$ there is chosen suitably (e.g. such that $\phi(x) = 0$ if $x \leq \delta/100$, and $\phi(x) = 1$ if $x \geq 99\delta/100$). So we will have the same bound for $|\p(X \in S) - \p(Z \in S)|$.
\end{proof}

\subsection{Normal comparison results}
In this subsection we record some celebrated results that let one compare probabilities for multivariate Gaussians with different covariance matrices.

\begin{normcomp1}[Slepian's Lemma, see e.g. Piterbarg~\cite{pit}]
Let $X=(X_{i})_{1 \leq i \leq n}$ and $W=(W_{i})_{1 \leq i \leq n}$ be vectors of jointly normal, mean zero, variance one random variables. Suppose that $\E X_i X_j \leq \E W_i W_j$ for all pairs $i,j$. Then for any real numbers $u_1 , ..., u_n$ we have
$$ \p(X_i \leq u_i \; \forall \; 1 \leq i \leq n) \leq \p(W_i \leq u_i \; \forall \; 1 \leq i \leq n) . $$
\end{normcomp1}

Slepian's Lemma says that decreasing the correlations of normal random variables makes them stochastically larger. It allows one to replace complicated correlations with simpler ones and it has the great advantage that it is never worse than trivial.

Using Slepian's Lemma we can prove the following useful bound.
\begin{normcomp2}
Suppose that $n \geq 2$ and that $\epsilon > 0$ is sufficiently small. Let $X_{1},...,X_{n}$ be mean zero, variance one, jointly normal random variables, and suppose that $\E X_{i}X_{j} \leq \epsilon$ whenever $i \neq j$. Then for any $A \geq 1$ and any $B> 0$ we have
$$ \p(\max_{1 \leq i \leq n} X_{i} \leq A) \ll \exp\left\{- \Theta\left(n \frac{e^{-A^2/2+O(\epsilon A^2 +AB+B^2)}}{A+B} \right) \right\} + e^{-B^{2}/\epsilon} . $$

In particular, for any $100\epsilon \leq \delta \leq 1/100$ (say) we have
$$ \p(\max_{1 \leq i \leq n} X_{i} \leq \sqrt{(2-\delta)\log n}) \ll e^{-\Theta(n^{\delta/20}/\sqrt{\log n})} + n^{-\delta^{2}/50\epsilon} . $$
\end{normcomp2}

\begin{proof}[Proof of Normal Comparison Result 2]
In view of Slepian's Lemma, the probability is at most as large as with the $X_i$ replaced by $W_i$, where $\E W_{i}W_{j} = \epsilon$ whenever $i \neq j$. However, it is well known that one can explicitly construct random variables with this covariance structure, by letting $Z_0 , Z_1 , ..., Z_n$ be independent standard normal random variables, and then taking
$$ W_{i} = \sqrt{\epsilon} Z_0 + \sqrt{1-\epsilon} Z_i \;\;\;\;\; \forall \; 1 \leq i \leq n . $$
So conditioning on the value of $Z_0$, we deduce that
$$ \p(\max_{1 \leq i \leq n} X_{i} \leq A) \leq \p(\max_{1 \leq i \leq n} W_{i} \leq A) = \frac{1}{\sqrt{2\pi}} \int_{-\infty}^{\infty} \left( \Phi\left(\frac{A + \sqrt{\epsilon}y}{\sqrt{1-\epsilon}}\right) \right)^{n} e^{-y^{2}/2} dy , $$
where $\Phi$ denotes the standard normal cumulative distribution function.

Splitting the integral at $y = B \sqrt{2/\epsilon}$, we find it is
\begin{equation}\label{BoundAB1}
 \leq \left( \Phi\left(\frac{A + B\sqrt{2}}{\sqrt{1-\epsilon}}\right) \right)^{n} + \int_{B \sqrt{2/\epsilon}}^{\infty} e^{-y^{2}/2} dy \ll \left( \Phi\left(\frac{A + B\sqrt{2}}{\sqrt{1-\epsilon}}\right) \right)^{n} + e^{-B^{2}/\epsilon} . 
 \end{equation}
Moreover, for any $z \geq 1$ we have
$$ \Phi(z) = 1 - \frac{1}{\sqrt{2\pi}} \int_{z}^{\infty} e^{-y^{2}/2} dy \leq 1 - \Theta(\frac{e^{-z^{2}/2}}{z}),$$
and so (since $\epsilon$ is small) we deduce
\begin{equation}\label{BoundAB2}
\left( \Phi\left(\frac{A + B\sqrt{2}}{\sqrt{1-\epsilon}}\right) \right)^{n} \leq \left(1-\Theta\left(\frac{e^{-(1+2\epsilon)(A+B\sqrt{2})^{2}/2}}{A+B}\right)\right)^{n}, 
\end{equation} from which 
the first bound claimed in Normal Comparison Result 2 follows.

The second bound follows on taking $A= \sqrt{(2-\delta)\log n}$ and $B = \delta\sqrt{(\log n)/50}$ in \eqref{BoundAB1} and \eqref{BoundAB2}, and noting then that
$$ n \frac{e^{-(1+2\epsilon)(A+B\sqrt{2})^{2}/2}}{A+B} \gg n \frac{e^{-(1+2\epsilon)A^{2}/2 - 2AB}}{A+B} \gg n \frac{e^{-(1+\delta/50)(1-\delta/2)\log n - (2/5)\delta\log n}}{\sqrt{\log n}}  \geq \frac{n^{\delta/20}}{\sqrt{\log n}} . $$
\end{proof}

The above results only give one sided bounds on probabilities, whereas in our theorems we want to show probabilities are equal up to a small error. The following result can supply such estimates in some cases.
\begin{normcomp3}[See Theorem 2.1 of Li and Shao~\cite{lishao}]
Let $X=(X_{i})_{1 \leq i \leq n}$ and $W=(W_{i})_{1 \leq i \leq n}$ be vectors of jointly normal, mean zero, variance one random variables. For each pair $i,j$ define $\rho_{i,j} := \max\{|\E X_i X_j |, |\E W_i W_j |\}$. Then for any real numbers $u_1 , ..., u_n$ we have
\begin{eqnarray}
&& \p(X_i \leq u_i \; \forall \; 1 \leq i \leq n) - \p(W_i \leq u_i \; \forall \; 1 \leq i \leq n) \nonumber \\
& \leq & \frac{1}{2\pi} \sum_{\substack{1 \leq i < j \leq n, \\ \E X_i X_j > \E W_i W_j}} (\textup{arcsin}(\E X_i X_j) - \textup{arcsin}(\E W_i W_j)) \exp\left\{- \frac{(u_i^{2} + u_j^{2})}{2(1+\rho_{i,j})} \right\} . \nonumber
\end{eqnarray}
\end{normcomp3}

Note that by swapping the roles of $X$ and $W$ one can obtain two sided bounds from this result.

A problem with Normal Comparison Result 3 is that the probabilities on the left may themselves be very small, and so the bound on the right may be worse than trivial. We know of no result that can overcome this difficulty in general, and a major issue in proving our theorems will be arranging things so that we only apply Normal Comparison Result 3 to probabilities that are fairly large, as discussed in the introduction.

\subsection{A result on leaders}\label{subsecprobleader}
In this subsection we shall establish Theorem \ref{probleader}. To this end, 
we shall first prove the following.
\begin{lem}\label{probleaderlemma}
Let the situation be as in Theorem \ref{probleader}. Then for any $x \geq 1$ we have
$$ |\p(\max_{2 \leq i \leq n} X_{i} \leq x | X_{1}=x) - \Phi(x)^{n-1}| \ll x e^{-(x^{2}/2)(1+O(\epsilon))} \sum_{2 \leq i \leq n} |r_{1,i}| + e^{-x^{2}(1+O(\epsilon))} \sum_{2 \leq i < j \leq n} |r_{i,j}| . $$
\end{lem}

\begin{proof}[Proof of Lemma \ref{probleaderlemma}]
Consider the transformed random variables
$$ V_{i} := \frac{X_{i}-r_{1,i}X_{1}}{\sqrt{1-r_{1,i}^{2}}}, \;\;\;\;\; 2 \leq i \leq n . $$
It is easy to check that these are all standard normal random variables, and they satisfy
$$ \E V_i X_1 = 0 \;\;\; \forall 2 \leq i \leq n, \;\;\;\;\; \E V_{i}V_{j} = \frac{r_{i,j}-r_{1,i}r_{1,j}}{\sqrt{(1-r_{1,i}^{2})(1-r_{1,j}^{2})}} = O(|r_{i,j}| + |r_{1,i}||r_{1,j}|) = O(\epsilon) \;\;\; \forall i \neq j . $$
In particular, since the $V_i$ are uncorrelated with $X_1$ they are independent from $X_1$, and so we see
$$ \p(\max_{2 \leq i \leq n} X_{i} \leq x | X_{1}=x) = \p\Biggl(V_{i} \leq \frac{x(1 - r_{1,i})}{\sqrt{1-r_{1,i}^{2}}} \; \forall 2 \leq i \leq n \Biggr) = \p\Biggl(V_{i} \leq x\sqrt{\frac{1-r_{1,i}}{1+r_{1,i}}} \; \forall 2 \leq i \leq n \Biggr) . $$

Next, using Normal Comparison Result 3 twice (with $X$ and $W$ swapped the second time) to compare the $V_i$ with independent standard normals we obtain
\begin{eqnarray}
&& \p\Biggl(V_{i} \leq x\sqrt{\frac{1-r_{1,i}}{1+r_{1,i}}} \; \forall 2 \leq i \leq n \Biggr) - \prod_{2 \leq i \leq n} \Phi\left(x\sqrt{\frac{1-r_{1,i}}{1+r_{1,i}}} \right) \nonumber \\
& = & O\left( \sum_{2 \leq i < j \leq n} |\E V_{i}V_{j}| \exp(-x^{2}\frac{((1-r_{1,i})/(1+r_{1,i})) + ((1-r_{1,j})/(1+r_{1,j}))}{2(1+|\E V_{i}V_{j}|)}) \right) \nonumber \\
& = & O\left( e^{-x^{2}(1+O(\epsilon))} \sum_{2 \leq i < j \leq n} (|r_{1,i}||r_{1,j}| + |r_{i,j}|) \right) . \nonumber
\end{eqnarray}
Notice here that the contribution from $|r_{i,j}|$ is acceptable for Lemma \ref{probleaderlemma}, and the contribution from $|r_{1,i}||r_{1,j}|$ may be rewritten as
$$ \left( e^{-x^{2}/2(1+O(\epsilon))} \sum_{2 \leq i \leq n} |r_{1,i}| \right)^{2} . $$
If this term is smaller than one then it is acceptable because it is smaller than the first error term in Lemma \ref{probleaderlemma}, and if it is bigger than one then so is the first error term in Lemma \ref{probleaderlemma}, so the lemma is trivially true.

To finish it will suffice to prove that
$$ \left| \prod_{2 \leq i \leq n} \Phi\left(x\sqrt{\frac{1-r_{1,i}}{1+r_{1,i}}} \right) - \Phi(x)^{n-1}\right| \ll x e^{-(x^{2}/2)(1+O(\epsilon))} \sum_{2 \leq i \leq n} |r_{1,i}| . $$
But this follows simply because we have
$$ \Phi\left(x\sqrt{\frac{1-r_{1,i}}{1+r_{1,i}}} \right) - \Phi(x) = \frac{1}{\sqrt{2\pi}} \int_{x}^{x\sqrt{(1-r_{1,i})/(1+r_{1,i})}} e^{-y^{2}/2} dy = O\left(x|r_{1,i}| e^{-(x^{2}/2)(1+O(\epsilon))} \right) $$
for each $2 \leq i \leq n$, on noting that $\sqrt{\frac{1-r_{1,i}}{1+r_{1,i}}} = 1 + O(|r_{1,i}|) = 1+ O(\epsilon)$ for all such $i$.
\end{proof}

\begin{proof}[Proof of Theorem \ref{probleader}]
Lemma \ref{probleaderlemma} looks very close to Theorem \ref{probleader}, but it doesn't immediately yield the theorem because it is quite weak (possibly worse than trivial) unless $x$ is fairly large. (As $x$ becomes smaller we expect the probabilities on the left to become very small, whereas the bound on the right becomes larger.) Fortunately we can deal with the case of small $x$ using Normal Comparison Result 2.

Indeed, if we choose $\delta = 1/1000$, say, then we clearly have
$$ |\p(X_{1} > \max_{2 \leq i \leq n} X_{i}) - \p(X_{1} > \max_{2 \leq i \leq n} X_{i}, \; \textrm{and} \, X_{1} > \sqrt{(2-\delta)\log n})| \leq \p(\max_{2 \leq i \leq n} X_{i} \leq \sqrt{(2-\delta)\log n}) , $$
and by Normal Comparison Result 2 the right hand side is
$$ \ll e^{-\Theta(n^{1/20000}/\sqrt{\log n})} + n^{-1/50000000\epsilon} \ll n^{-100} , $$
since $\epsilon$ is assumed to be small enough. We would have the same bound if the $X_i$ were replaced by independent normals $\tilde{X_i}$, and in that case we have $\p(\tilde{X_{1}} > \max_{2 \leq i \leq n} \tilde{X_{i}}) = 1/n$ by symmetry, so to prove Theorem \ref{probleader} it will suffice to show that
\begin{eqnarray}
&& |\p(X_{1} > \max_{2 \leq i \leq n} X_{i}, \; \textrm{and} \, X_{1} > \sqrt{(2-\delta)\log n}) - \p(\tilde{X_{1}} > \max_{2 \leq i \leq n} \tilde{X_{i}}, \; \textrm{and} \, \tilde{X_{1}} > \sqrt{(2-\delta)\log n}) | \nonumber \\
& \ll & n^{-1.99}\sum_{2 \leq i \leq n} |r_{1,i}| + n^{-2.99} \sum_{2 \leq i < j \leq n} |r_{i,j}| . \nonumber
\end{eqnarray}

But using Lemma \ref{probleaderlemma} we have
\begin{eqnarray}
&& \p(X_{1} > \max_{2 \leq i \leq n} X_{i}, \; \textrm{and} \; X_{1} > \sqrt{(2-\delta)\log n}) - \p(\tilde{X_{1}} > \max_{2 \leq i \leq n} \tilde{X_{i}}, \; \textrm{and} \, \tilde{X_{1}} > \sqrt{(2-\delta)\log n}) \nonumber \\
& = & \int_{\sqrt{(2-\delta)\log n}}^{\infty} \frac{e^{-x^{2}/2}}{\sqrt{2\pi}} \p(\max_{2 \leq i \leq n} X_{i} \leq x | X_{1}=x) dx - \int_{\sqrt{(2-\delta)\log n}}^{\infty} \frac{e^{-x^{2}/2}}{\sqrt{2\pi}} \Phi(x)^{n-1} dx  \nonumber \\
& = & O\left(n^{-(2-\delta)(1+O(\epsilon))}\sum_{2 \leq i \leq n} |r_{1,i}| + n^{-(3/2)(2-\delta)(1+O(\epsilon))} \sum_{2 \leq i < j \leq n} |r_{i,j}| \right) . \nonumber
\end{eqnarray}
The Theorem follows on remembering that $\delta = 1/1000$ and $\epsilon$ is sufficiently small.
\end{proof}

\subsection{Preparation for the $k$-contestant theorems}
In this subsection we shall develop two lemmas that will be required later for the proof of Theorem \ref{OrderFirstk1}. We present these in a moderate amount of generality, and they might be of wider interest.

We begin with a kind of hybrid normal comparison inequality.
\begin{lem}\label{hybridcomparison}
Let $(W_i)_{i \in \mathcal{I}}$ be a finite collection of jointly standard normal random variables, let $\epsilon \geq \epsilon_1 > 0$ be small, and suppose that
$$ |\E W_i W_j| \leq \epsilon \; \forall i \neq j, \;\;\; \text{and} \;\;\; \E W_i W_j \leq \epsilon_1 \; \forall i \neq j . $$

Then if $w \geq 1$, and $w/2 \leq w_i \leq 2w$ for all $i \in \mathcal{I}$, we have
\begin{eqnarray}
&& \left| \p(W_i \leq w_i \; \forall i \in \mathcal{I}) - \prod_{i \in \mathcal{I}} \Phi(w_i) \right| \nonumber \\
& \ll & \left(\prod_{i \in \mathcal{I}} \Phi((1+O(\epsilon))w_i) + e^{-\Theta((\epsilon w)^2/(\epsilon_1 + \epsilon^3))} \right) \sum_{\substack{i,j \in \mathcal{I}, \\ i \neq j}} |\E W_i W_j| e^{-(1/2+O(\epsilon))(w_i^2 + w_j^2)} . \nonumber
\end{eqnarray}
\end{lem}

The key advantage of Lemma \ref{hybridcomparison}, as opposed to a result like Normal Comparison Result 3, is the presence of the bracketed prefactor on the right hand side, which can make the inequality much more powerful if $\prod_{i \in \mathcal{I}} \Phi(w_i)$ is small. Notice also that if one has a stronger upper bound for $\E W_i W_j$ than an absolute value bound (i.e. if $\epsilon_1$ is appreciably smaller than $\epsilon$), then the second term in the bracket is improved. This should not be too surprising, since Slepian's Lemma (Normal Comparison Result 1) implies that negative correlations make normal random variables stochastically smaller.

\begin{proof}[Proof of Lemma \ref{hybridcomparison}]
If one looks inside the proofs of the usual normal comparison lemmas (see e.g. Li and Shao's paper~\cite{lishao}), they yield that
\begin{eqnarray}
|\p(W_i \leq w_i \; \forall i \in \mathcal{I}) - \prod_{i \in \mathcal{I}} \Phi(w_i)| & \ll & \sum_{\substack{i,j \in \mathcal{I}, \\ i \neq j}} P_{i,j} |\E W_i W_j| e^{-(w_i^2 + w_j^2)/2(1+|\E W_i W_j|)} \nonumber \\
& = & \sum_{\substack{i,j \in \mathcal{I}, \\ i \neq j}} P_{i,j} |\E W_i W_j| e^{-(1/2+O(\epsilon))(w_i^2 + w_j^2)} , \nonumber
\end{eqnarray}
where
$$ P_{i,j} := \sup_{0 \leq h \leq 1} \p(W_{k}^{(h)} \leq w_k \, \forall \, k \in \mathcal{I} | W_{i}^{(h)}= w_i, \, W_{j}^{(h)}= w_j) , $$
and the $W_{k}^{(h)}$ are mean zero, variance one, jointly normal random variables with correlations
$$ \E W_{k}^{(h)}W_{l}^{(h)} = h \E W_k W_l, \,\,\,\,\, k \neq l . $$
(Thus, in particular, when $h=1$ the $W_{k}^{(h)}$ are simply the $W_{k}$, and when $h=0$ the $W_{k}^{(h)}$ are independent standard normal random variables.) We will show that
$$ \p(W_{k} \leq w_k \, \forall \, k \in \mathcal{I} | W_{i} = w_i, \, W_{j} = w_j) \ll \prod_{i \in \mathcal{I}} \Phi((1+O(\epsilon))w_i) + e^{-\Theta((\epsilon w)^2/(\epsilon_1 + \epsilon^3))} $$
for any $i \neq j$. Exactly the same argument would yield the corresponding estimate for any $0 \leq h \leq 1$ (uniformly over $h$), thus proving Lemma \ref{hybridcomparison}.

Indeed, if we define
$$ V_k := W_k - \frac{(\E W_k W_i - \E W_k W_j \E W_i W_j)}{1 - (\E W_i W_j)^2} W_i - \frac{(\E W_k W_j - \E W_k W_i \E W_i W_j)}{1 - (\E W_i W_j)^2} W_j \;\;\; \forall k \in \mathcal{I}\backslash \{i,j\}, $$
then one can check that $\E V_k W_i = \E V_k W_j = 0$, and so if $k, l \in \mathcal{I} \backslash \{i,j\}$ then
$$ \E V_k V_l = \E V_k W_l  - 0 - 0 = \E V_k W_l = \E W_k W_l - \E W_k W_i \E W_i W_l - \E W_k W_j \E W_j W_l + O(\epsilon^3) . $$
In particular, it follows that $\E V_{k}^{2} = 1 + O(\epsilon^{2})$. One can also check that when $k \neq l$ we have the upper bound
$$ \E V_{k} V_{l} \leq \epsilon_1 + 2\epsilon \epsilon_1 + O(\epsilon^{3}) \leq (C/2)(\epsilon_{1} + \epsilon^{3}) , $$
for a suitable positive constant $C$, since $\E W_k W_l \leq \epsilon_1$ and if $- \E W_k W_i \E W_i W_l$ is positive then one of the factors must be positive and one negative, so one has size at most $\epsilon_1$ and the other has size at most $\epsilon$.

Now if we set $\tilde{V_{k}} := \frac{V_{k}}{\sqrt{\E V_{k}^{2}}}$, then $\p(W_{k} \leq w_k \, \forall \, k \in \mathcal{I} | W_{i} = w_i, \, W_{j} = w_j)$ is
\begin{eqnarray}
& = & \p(V_{k} \leq w_k - \frac{(\E W_k W_i - \E W_k W_j \E W_i W_j)}{1 - (\E W_i W_j)^2} w_i - \frac{(\E W_k W_j - \E W_k W_i \E W_i W_j)}{1 - (\E W_i W_j)^2} w_j \, \forall \, k \in \mathcal{I} \backslash \{i,j\}) \nonumber \\
& = & \p(\tilde{V_{k}} \leq (1+O(\epsilon))w_{k} \; \forall \, k \in \mathcal{I} \backslash \{i,j\}) , \nonumber
\end{eqnarray}
where the first equality uses the fact that $V_{k}$ is uncorrelated with, and therefore independent of, $W_{i}$ and $W_{j}$ (so we can remove the conditioning), and the second equality uses the fact that $w_{i}, w_{j} \asymp w \asymp w_{k}$.

Now the $\tilde{V_{k}}$ are jointly standard normal random variables with off-diagonal correlations that are $\leq (1+O(\epsilon^{2}))\E V_{k} V_{l} \leq C (\epsilon_{1} + \epsilon^{3})$, so by Slepian's Lemma (Normal Comparison Result 1) the probability is at most as large as if all the off-diagonal correlations were equal to $C(\epsilon_{1} + \epsilon^{3})$. Arguing as in the proof of Normal Comparison Result 2, it follows that
\begin{align*}
&\p(\tilde{V_{k}} \leq (1+O(\epsilon))w_{k} \; \forall \, k \in \mathcal{I} \backslash \{i,j\})\\
& \leq \frac{1}{\sqrt{2\pi}} \int_{-\infty}^{\infty} \prod_{\substack{k \in \mathcal{I}, \\ k \neq i,j}} \Phi\left(\frac{(1+O(\epsilon))w_{k} + \sqrt{C(\epsilon_{1} + \epsilon^{3})}y}{\sqrt{1-C(\epsilon_{1} + \epsilon^{3})}}\right) e^{-y^{2}/2} dy ,
\end{align*}
and Lemma \ref{hybridcomparison} follows on splitting the integral at $y = \epsilon w/\sqrt{C(\epsilon_{1} + \epsilon^{3})}$.
\end{proof}

We finish this section by applying Lemma \ref{hybridcomparison} to prove the following general estimate for $\p(\max_{k+1 \leq i \leq n} X_{i} \leq x_k | X_{1}=x_1 , ..., X_{k} = x_k)$, which is what we shall actually use later when proving Theorem \ref{OrderFirstk1}.
\begin{lem}\label{lemforfirstk}
Let $X_{1},...,X_{n}$ be mean zero, variance one, jointly normal random variables, and write $r_{i,j} := \E X_{i}X_{j}$. Let $\epsilon \geq \epsilon_{1}, \epsilon_{2} > 0$ be sufficiently small, and suppose that
$$ |r_{i,j}| \leq \epsilon \; \forall i \neq j, \;\;\; \text{and} \;\;\; r_{i,j} \leq \epsilon_{1} \; \forall i \neq j . $$
Let $1 \leq k \leq n-2$ be such that $\epsilon k$ is sufficiently small, and suppose further that
$$  \sum_{l=1}^{k} \sum_{s=1, s \neq l}^{k} |r_{i,l}| |r_{j,s}| \Bigg( |r_{l,s}|+\sum_{t = 1, t \neq l,s}^{k} |r_{l,t}||r_{s,t}| + \epsilon^2 \Big(\sum_{\substack{1 \leq t,u \leq k, \\ t \neq u}} |r_{t,u}| \Big) \Bigg) \leq \epsilon_2 $$
for any distinct $k + 1 \leq i,j \leq n$.

Then for any real numbers $x_{1}, ..., x_{k-1} \geq x_{k} \geq 1$ such that $\sum_{i=1}^{k} x_i^2 \ll 1/\epsilon$, we have
\begin{eqnarray}
&& \left| \p(\max_{k+1 \leq i \leq n} X_{i} \leq x_k | X_{1}=x_1 , ..., X_{k} = x_k) - \prod_{i=k+1}^{n} \Phi\left(w_i \right) \right| \nonumber \\
& \ll & e^{-\left(1+O(\sqrt{\epsilon k})\right)x_k^2} \left(\prod_{i=k+1}^{n} \Phi\left((1+O(\epsilon)) w_i \right) + e^{-\Theta\big((\epsilon x_k)^2/(\epsilon_1 + \epsilon_2 + \epsilon^3)\big)} \right) \\
& &\quad \quad  \quad \quad \quad \times  \left( \sum_{k+1 \leq i < j \leq n} |r_{i,j}| + \sum_{l=1}^{k} \left(\sum_{i=k+1}^{n} |r_{i,l}| \right)^2 \right) , \nonumber
\end{eqnarray}
for certain numbers $w_i$ that satisfy
$$w_i= (1+O(\epsilon^2 k))x_k + O\Big(\sum_{l=1}^{k} x_l |r_{i,l}| \Big)=  \big(1+O\big(\sqrt{\epsilon k}\big)\big)x_k. $$
\end{lem}

The reader may wish to compare Lemma \ref{lemforfirstk} with Lemma \ref{probleaderlemma}. The first bracketed term on the right hand side of Lemma \ref{lemforfirstk} will be crucial when we come to apply the lemma to prime number races with $k$ large, since on the relevant range of $x_1, ..., x_k$ it will turn out that $\prod_{i=k+1}^{n} \Phi\left(w_i \right)$ is rather small (in fact of size roughly $e^{-Ck}$).

\begin{proof}[Proof of Lemma \ref{lemforfirstk}]
Note first that provided $\epsilon k \leq 1/2$, Lemma \ref{matrixlem} implies that the covariance (sub-)matrix $\mathcal{A} := (r_{i,j})_{1 \leq i,j \leq k}$ is invertible. Let $(\tilde{r}_{i,j})_{1 \leq i,j \leq k} := \mathcal{A}^{-1}$ denote the inverse matrix, and for all $k+1 \leq i \leq n$ and all $1 \leq l \leq k$ set
$$ u_{i, l} := \sum_{s=1}^k \widetilde{r}_{l,s} r_{i,s} , \;\;\; \text{and} \;\;\; V_{i} := X_{i}-\sum_{l=1}^k u_{i, l} X_l . $$
Then the random variables $V_{i}$ are zero mean, jointly normal random variables, and for any $1 \leq t \leq k$ they satisfy
$$ \E V_{i} X_{t} = r_{i,t} - \sum_{l=1}^{k} u_{i,l} r_{l,t} = r_{i,t} - \sum_{1 \leq l,s \leq k} r_{i,s} \tilde{r}_{l,s} r_{l,t} = r_{i,t} - \sum_{1 \leq s \leq k} r_{i,s} \textbf{1}_{s=t} = 0 . $$
In other words, the random variables $V_{i}$ are uncorrelated with, and therefore independent from, all of $X_1, X_2, \dots, X_k$. Let us also note that for any $k+1 \leq i \leq n$ and any $1 \leq l \leq k$, using Lemma \ref{matrixlem} to estimate the $\widetilde{r}_{l,s}$ yields that
\begin{equation}\label{sizecoeff}
 u_{i,l} = \Big(1 + O\Big(\epsilon \sum_{\substack{1 \leq t,u \leq k, \\ t \neq u}} |r_{t,u}|\Big)\Big) r_{i,l} + O\left(\sum_{\substack{s=1, \\ s \neq l}}^k |r_{i,s}| \left(|r_{l,s}| + \sum_{t \neq l,s} |r_{l,t}||r_{s,t}| + \epsilon^{2} \sum_{\substack{1 \leq t,u \leq k, \\ t \neq u}} |r_{t,u}|\right) \right),
 \end{equation}
and using our assumptions that the off-diagonal covariances are bounded by $\epsilon$, and that $\epsilon k$ is small, it follows in particular that 
\begin{equation}\label{sizecoeff2}
|u_{i,l}| \ll |r_{i,l}| + \epsilon \sum_{\substack{s=1, \\ s \neq l}}^{k} |r_{i,s}| \ll \epsilon.
\end{equation}
 We also see that
\begin{equation}\label{sizecoeff3}
 \E V_{i}^{2} = \E V_{i}X_{i} - 0 = \E V_{i}X_{i} = 1 + O\left(\epsilon \sum_{l=1}^{k} |u_{i,l}|\right) = 1 + O(\epsilon^2 k) . 
 \end{equation}

In view of the above discussion, if we rewrite the event $\max_{k+1 \leq i \leq n} X_{i} \leq x_k$ in terms of the $V_{i}$ so that we can remove the conditioning, we find
\begin{align*}
 \p(\max_{k+1 \leq i \leq n} X_{i} \leq x_k | X_{1}=x_1 , ..., X_{k} = x_k)
&= \p\left(V_{i} \leq x_{k} -\sum_{l=1}^{k} x_l u_{i,l} \; \forall k+1 \leq i \leq n \right) \\
&= \p\left(W_{i} \leq w_i\;  \forall k+1 \leq i \leq n \right) 
\end{align*}
where $w_i:= \frac{1}{{\sqrt{\E V_{i}^{2}}}}\Big(x_{k} -\sum_{l=1}^{k} x_l u_{i,l}\Big)$, and  $W_i := \frac{V_{i}}{\sqrt{\E V_{i}^{2}}}$ are mean zero, variance one, jointly normal random variables. Proceeding to estimate the off-diagonal correlations, we note first that
$$\E W_{i}W_{j} = (1+O(\epsilon^2 k))\E V_{i} V_{j} = (1+O(\epsilon^2 k)) \E X_{i}V_{j} , $$ 
and then from \eqref{sizecoeff} we deduce that 
\begin{equation}\label{SizeCorrW}
\begin{aligned}
& \E W_{i}W_{j}= (1+O(\epsilon^2 k))r_{i,j} - \Big(1+O\Big(\epsilon^2 k + \epsilon \sum_{\substack{1 \leq t,u \leq k, \\ t \neq u}} |r_{t,u}|\Big)\Big)\sum_{l=1}^{k} r_{i,l} r_{j,l}  \\
& + O\left( \sum_{l=1}^{k} \sum_{s=1, s \neq l}^{k} |r_{i,l}| |r_{j,s}| \Bigg(|r_{l,s}|+ \sum_{t = 1, t \neq l,s}^{k} |r_{l,t}||r_{s,t}| + \epsilon^2 \Big(\sum_{\substack{1 \leq t,u \leq k, \\ t \neq u}} |r_{t,u}| \Big) \Bigg)\right).
\end{aligned}
\end{equation}
Since $|r_{i,j}| \leq \epsilon$ and $r_{i,j} \leq \epsilon_{1}$ when $i \neq j$, and since $\epsilon k$ is small, it follows in particular that whenever $i \neq j$ we have $|\E W_i W_j| = O(\epsilon)$, and
$$
 \E W_i W_j \leq (1+O(\epsilon^2 k))\epsilon_1 + (1+O(\epsilon^{2}k^2))k\epsilon \epsilon_1 + O(\epsilon_2) = O(\epsilon_1 + \epsilon_2) .
$$

Finally, using that  $\sum_{l=1}^k x_l\leq \sqrt{k\sum_{l=1}^{k} x_l^2}\ll \sqrt{k/\epsilon}$ (which follows from the Cauchy--Schwarz inequality) together with  \eqref{sizecoeff2} we find
\begin{align*}
w_i=(1+O(\epsilon^2 k))x_k + O\Big(\sum_{l=1}^{k} x_l |r_{i,l}|+ \sqrt{\epsilon k}\sum_{s=1}^{k} |r_{i,s}|\Big)
= \big(1+O(\epsilon^2 k)\big)x_k + O\left(\sum_{l=1}^{k} x_l |r_{i,l}| \right).
\end{align*}
Here the final equality follows because the $x_l$ are all $\geq 1$ whereas $\epsilon k$ is small. We also have $\sum_{l=1}^{k} x_l |r_{i,l}| \leq \epsilon \sum_{l=1}^{k} x_l \ll \sqrt{\epsilon k}$, and therefore $w_i=  \big(1+O\big(\sqrt{\epsilon k}\big)\big)x_k$, and hence $ x_k/2 \leq w_i \leq 2x_k $ for all $k+1 \leq i \leq n$.

This all means that Lemma \ref{hybridcomparison} is applicable, with $\epsilon_1$ replaced by $O(\epsilon_1 + \epsilon_2)$ and $w$ replaced by $x_k$. Lemma \ref{lemforfirstk} follows from Lemma \ref{hybridcomparison} on noting that 
$$
\big(1/2+O(\epsilon)\big)(w_i^2 + w_j^2) = \big(1+O(\sqrt{\epsilon k})\big)x_k^2,
$$ and using  \eqref{SizeCorrW} to get\begin{align*}
\sum_{k+1 \leq i < j \leq n} |\E W_i W_j| 
& \ll  \sum_{k+1 \leq i < j \leq n} \left(|r_{i,j}| + \sum_{l=1}^{k} |r_{i,l}| |r_{j,l}| + \epsilon \sum_{l=1}^{k} \sum_{s=1, s \neq l}^{k} |r_{i,l}| |r_{j,s}| \right)  \\
& \ll \sum_{k+1 \leq i < j \leq n} |r_{i,j}| + \sum_{l=1}^{k} \left(\sum_{i=k+1}^{n} |r_{i,l}| \right)^2+ \epsilon \left(\sum_{l=1}^{k} \sum_{i=k+1}^{n} |r_{i,l}| \right)^2   \\
& \ll  \sum_{k+1 \leq i < j \leq n} |r_{i,j}| + \sum_{l=1}^{k} \left(\sum_{i=k+1}^{n} |r_{i,l}| \right)^2 . 
\end{align*}
Here the final line uses the Cauchy--Schwarz inequality in the form $\left(\sum_{l=1}^{k} \sum_{i=k+1}^{n} |r_{i,l}| \right)^2 \leq k \sum_{l=1}^{k} \left(\sum_{i=k+1}^{n} |r_{i,l}| \right)^2$.
\end{proof}

\section{The full prime number race: Proof of Theorem \ref{FullRace}}\label{secFullRace}
Suppose $q$ is large, and let $Z=(Z_{j})_{1 \leq j \leq n}$ denote a multivariate normal random vector whose components have mean zero, variance one, and correlations
$$ r_{i, j}=\E Z_{i}Z_{j} = \frac{B_{q}(a_i,a_j)}{\var(q)} . $$
Let $S_1, S_2 \subseteq \{1, 2, \dots, n\}$ be non-empty. Then Correlation Estimate 1 implies that
\begin{equation}\label{AverageCorr}
\sum_{\substack{\ell\in S_1, s\in S_2\\l\neq s}}|r_{\ell,s}|\ll \frac{\sqrt{|S_1|\cdot |S_2|}\log^2 \big(2|S_1|\cdot |S_2|\big)}{\log q} , \;\;\; \text{and} \;\;\; \sum_{\substack{1 \leq \ell, s \leq n, \\ \ell \neq s}} |r_{\ell,s}| \ll \frac{n \log^{2}(2n)}{\log q} .
\end{equation}

Let $\mathcal{C}=(r_{i,j})_{1\leq i,j\leq n}$ be the covariance matrix of $Z_1, \dots, Z_n$. In view of the pointwise bound $|r_{i,j}| \ll 1/\log q$ (for $i \neq j$) from Lemma \ref{CovarianceEntries}, we may apply Lemma \ref{matrixlem} with $\epsilon \asymp 1/\log q$ provided that $n$ is at most a certain constant times $\log q$, and obtain that
\begin{equation}\label{determinant}
\det(\mathcal{C})=1+O\left(\frac{n(\log n)^2}{(\log q)^2}\right),
\end{equation}
and 
\begin{equation}\label{estimateR}
\widetilde{r}_{\ell,s} = \begin{cases}  1+ O\left(\frac{n(\log n)^2}{(\log q)^2}\right) & \text{ if } \ell=s\\
O \left(|r_{\ell,s}|+ \sum_{\substack{1\leq j\leq n\\ j\neq \ell, s}}|r_{\ell, j}||r_{s, j}|+\frac{n(\log n)^2}{(\log q)^3}\right) &  \text{ if } \ell\neq s ,
\end{cases} 
\end{equation}
where $\widetilde{r}_{i,j}$ denote the entries of the inverse matrix $\mathcal{C}^{-1}$.

Now the key ingredient in the proof of Theorem \ref{FullRace} is the following proposition, which gives an approximation for the joint density function of $Z_1, \dots, Z_n$. Let $f(x_1, \dots, x_n)$ be this density, and define the Euclidean norm $||x||:=(x_1^2+\cdots+ x_n^2)^{1/2}$.

\begin{pro}\label{Density}
Let $2 \leq n \ll \log q$ be an integer. For any $x=(x_1, \dots, x_n)\in \mathbb{R}^n$, we have
\begin{align*}
f(x_1, \dots, x_n)= &\left(1+O\left(\frac{n(\log n)^2}{(\log q)^2}\right)\right)\\
&\times \frac{1}{(2\pi)^{n/2}}\exp\left(-\frac{||x||^2}{2}\left(1+O\left(\frac{(\log n)^4}{\log q}+\frac{n(\log n)^6}{(\log q)^2}+ \frac{n^2(\log n)^2}{(\log q)^3}\right)\right)\right).
\end{align*}

\end{pro}
\begin{proof}
By definition and by \eqref{determinant} we have
\begin{equation}\label{DefDensity}
\begin{aligned}f(x_1, \dots, x_n)&=\frac{1}{(2\pi)^{n/2}\sqrt{\text{det}(\mathcal{C})}}\exp\left(-\frac12 \mathbf{x}^T \mathcal{C}^{-1} \mathbf{x}\right)\\
&=  \left(1+O\left(\frac{n(\log n)^2}{(\log q)^2}\right)\right)\frac{1}{(2\pi)^{n/2}}\exp\left(-\frac12 \mathbf{x}^T \mathcal{C}^{-1} \mathbf{x}\right) ,
\end{aligned}
\end{equation}
and using \eqref{estimateR} we obtain 
\begin{equation}\label{estimatedensity}
\mathbf{x}^T \mathcal{C}^{-1} \mathbf{x}= \sum_{1\leq \ell, s\leq n} x_{\ell}x_s \widetilde{r}_{\ell,s}= ||x||^2 \left(1+O\left(\frac{n(\log n)^2}{(\log q)^2}\right)\right)+ \sum_{1\leq \ell\neq s\leq n} x_{\ell}x_s \widetilde{r}_{\ell,s} ,
\end{equation}
and also that
\begin{equation}\label{BoundOFF}
\sum_{1\leq \ell\neq s\leq n} x_{\ell}x_s \widetilde{r}_{\ell,s}\ll \sum_{1\leq \ell\neq s\leq n} |x_{\ell}x_sr_{\ell,s}| + \sum_{1\leq \ell\neq s\leq n} |x_{\ell}x_s|\sum_{\substack{1\leq j\leq n\\ j\neq \ell, s}}|r_{\ell, j}||r_{s, j}|+ \frac{n(\log n)^2}{(\log q)^3}\sum_{1\leq \ell\neq s\leq n} |x_{\ell}x_s|.
\end{equation}

Now let $J := \lfloor 2\log n \rfloor$, and for each $1\leq j \leq J-1$ define $S_j$ to be the subset of $\{1,\dots, n\}$ consisting of those $\ell$ for which $||x||/2^j < |x_{\ell}|\leq ||x||/2^{j-1}$. Also let $S_J$ be the set of those $\ell$ for which $|x_{\ell}|\leq  ||x||/2^{J-1}$. Then note that for all $1\leq j\leq  J$ we have $|S_j|\leq 2^{2j}$, since $|S_J|\leq n\leq 2^{2J}$ and if $j<J$, then
$ \frac{||x||^2}{2^{2j}}|S_j| \leq x_1^2+ \cdots x_n^2=||x||^2$.  Using \eqref{AverageCorr} we deduce that
\begin{align*}
\sum_{1\leq \ell\neq s\leq n} |x_{\ell}x_sr_{\ell,s}| = \sum_{1\leq i, j\leq J} \sum_{\substack{\ell\in S_i, s\in S_j\\ \ell\neq s}}|x_{\ell}x_sr_{\ell,s}| & \ll ||x||^2\sum_{1\leq i, j\leq J} \frac{1}{2^{i+j}} \sum_{\substack{\ell\in S_i, s\in S_j\\ \ell\neq s}}|r_{\ell,s}|
\\
&\ll \frac{(\log n)^4}{\log q}||x||^2.
\end{align*}
Similarly we have
\begin{eqnarray}
\sum_{1\leq \ell\neq s\leq n} |x_{\ell}x_s|\sum_{\substack{1\leq j\leq n\\ j\neq \ell, s}}|r_{\ell, j}||r_{s, j}| \leq \sum_{1 \leq j \leq n} \left(\sum_{\substack{1\leq \ell \leq n, \\ \ell \neq j}} |x_{\ell}| |r_{\ell, j}| \right)^{2} & \ll & \sum_{1 \leq j \leq n} \left(\sum_{1 \leq i \leq J} \frac{||x||}{2^{i}} \sum_{\substack{\ell \in S_i, \\ \ell \neq j}} |r_{\ell, j}| \right)^{2} \nonumber \\
& \ll & \frac{n(\log n)^6}{(\log q)^2}||x||^2 , \nonumber
\end{eqnarray}
since \eqref{AverageCorr} implies that $\sum_{\substack{\ell \in S_i, \\ \ell \neq j}} |r_{\ell, j}| \ll (\sqrt{|S_i|} \log^{2}(2|S_i|))/\log q \ll 2^{i} i^{2}/\log q$.
Finally, by the Cauchy--Schwarz inequality we have $\sum_{1\leq \ell\neq s\leq n} |x_{\ell}x_s|\leq n||x||^2$, and hence the contribution of the third term in the right hand side of \eqref{BoundOFF} is $\ll n^2(\log n)^2 ||x||^{2}/(\log q)^3$. Collecting the above estimates completes the proof.
\end{proof}

\begin{proof}[Proof of Theorem \ref{FullRace}]
By Proposition \ref{Density}, if $2\leq n \leq \log q/(\log\log q)^4$ then
\begin{align*}
&\pr(Z_1>Z_2>\cdots> Z_n) = \int_{x_1>\cdots>x_n}f(x_1,  \dots, x_n)dx_1\cdots dx_n\\
&= \left(1+O\left(\frac{n(\log n)^2}{(\log q)^2}\right)\right)\frac{1}{(2\pi)^{n/2}}\int_{x_1>\cdots>x_n}\exp\left(-\frac{||x||^2}{2}\left(1+O\left(\frac{(\log n)^4}{\log q}\right)\right)\right)dx_1\cdots dx_n\\
&= \left(1+O\left(\frac{n(\log n)^2}{(\log q)^2}\right)\right)\frac{1}{n!}\left(\frac{1}{\sqrt{2\pi}} \int_{-\infty}^{\infty} \exp\left(-\frac{t^2}{2}\left(1+O\left(\frac{(\log n)^4}{\log q}\right)\right)\right)dt\right)^n \\
&= \left(1+O\left(\frac{n(\log n)^4}{\log q}\right)\right)\frac{1}{n!},
\end{align*}
since the integrand is symmetric in $x_1, \dots, x_n$.

Combining this with Normal Approximation Result 1, we obtain
$$ \pr(X(q,a_1) > X(q,a_2) > \cdots > X(q,a_n)) = \left(1+O\left(\frac{n(\log n)^4}{\log q}\right)\right)\frac{1}{n!} + O\left(\frac{n^{3}}{\varphi(q)^{1/8}} \right) . $$
If $n \leq \log q/(\log\log q)^4$ then $1/n! \geq 1/n^{n} = 1/q^{o(1)}$, so the second ``big Oh'' term may be absorbed into the first, and Theorem \ref{FullRace} follows.
\end{proof}

\section{The leader: Proof of Theorem \ref{Leader}}\label{secLeader}
We may assume that $q$ is sufficiently large, otherwise the theorem is trivial. Then we may assume that $n \geq \log^{0.1}q$, say, otherwise the theorem follows from Theorem \ref{FullRace}. Let $Z=(Z_{j})_{1 \leq j \leq n}$ denote a multivariate normal random vector whose components have mean zero, variance one, and correlations
$ r_{i, j}=\E Z_{i}Z_{j} = B_{q}(a_i,a_j)/\var(q). $

In view of Theorem \ref{probleader} and Correlation Estimate 1, and our assumption that $n \geq \log^{0.1}q$,
\begin{eqnarray}
|\p(Z_{1} > \max_{2 \leq i \leq n} Z_{i}) - \frac{1}{n}| & \ll & n^{-100} + n^{-1.99}\sum_{2 \leq i \leq n} |r_{1,i}| + n^{-2.99} \sum_{2 \leq i < j \leq n} |r_{i,j}| \nonumber \\
& \ll & n^{-100} + \frac{\sqrt{n} \log^{2}n}{n^{1.99} \log q} + \frac{n \log^{2}n}{n^{2.99} \log q} \nonumber \\
& \ll & \frac{\log^{2}n}{n^{1.49} \log q} = \frac{1}{n} \cdot \frac{\log^{2}n}{n^{0.49} \log q}  . \nonumber
\end{eqnarray}
Theorem \ref{Leader} follows by combining this estimate with Normal Approximation Result 1.

\section{Ordering the first $k$ contestants: Proof of Theorem \ref{OrderFirstk1} }\label{secorderfirstk}

Throughout this section we let $Z_{1}, ..., Z_{n}$ denote mean zero, variance one, jointly normal random variables corresponding to a prime number race modulo $q$ (i.e. with off-diagonal correlations $r_{i,j} = \E Z_{i}Z_{j} = B_{q}(a_{i},a_{j})/\text{Var}(q)$), where $q$ is large.

Our key tool in proving Theorem \ref{OrderFirstk1} will be the version of Lemma \ref{lemforfirstk} that arises from specializing to the prime number race situation. We record this now, together with an estimate for the number of large values of the $|r_{i,j}|$ that we shall need when deducing it.

\begin{lem}\label{DistribCorr}
For any $k+1 \leq i \leq n$ and any $j \geq 1$ we have
$$ \#\left\{1 \leq l \leq k: |r_{i,l}| \geq \frac{1}{2^{j} \log q}\right\} \ll 2^{2j} j^{4}. $$
\end{lem}

\begin{proof}
Let $\mathcal{S}_{j}$ denote the set of $1 \leq l \leq k$ for which $|r_{i,l}| \geq \frac{1}{2^{j} \log q}$. Using Correlation Estimate 1, we have
$$ \sum_{l \in \mathcal{S}_{j}} |r_{i,l}| \ll \frac{\sqrt{\#\mathcal{S}_{j}} \log^{2}(2\#\mathcal{S}_{j})}{\log q}. $$
On the other hand, by definition of $\mathcal{S}_{j}$ the left hand side is $\geq \frac{\#\mathcal{S}_{j}}{2^{j} \log q}$, and so the lemma follows by rearranging.
\end{proof}

\begin{lem}\label{lemfirstkPNR}
Let $k$ be a positive integer such that $k/\log q$ is small enough, and suppose $n \geq k+2$ is large. Let $1\leq A\leq 2\sqrt{\log n}$ be real. If $x= (x_{1}, ..., x_k)$ is such that $x_{1}, ..., x_{k-1} \geq x_{k} \geq A$ and $||x|| \leq 10\sqrt{\log q}$, then we have
\begin{align*}
 &\p(\max_{k+1 \leq i \leq n} Z_{i} \leq x_k | Z_{1}=x_1 , ..., Z_{k} = x_k)
= \prod_{i=k+1}^{n} \Phi\left(w_i \right)\\ 
&+ O\left(e^{-A^2 \Big(1+O\big(\sqrt{k/\log q}\big)\Big)}\frac{n(\log n)^4}{\log q} \left(\prod_{i=k+1}^{n} \Phi\left((1+O(1/\log q))w_i \right) + e^{-\Theta(A^2 (\log q)/(\log\log q)^{9})} \right) \right) , 
\end{align*}
for certain numbers $w_i$ that satisfy
$$w_i= \left(1+O\left(\frac{k}{(\log q)^2}\right)\right)x_k +O\left(\sum_{l=1}^{k} x_l|r_{i,l}| \right).$$
\end{lem}

\begin{proof}
We want to apply Lemma \ref{lemforfirstk}, and Lemma \ref{CovarianceEntries} shows that we may do so with $\epsilon \asymp 1/\log q$. We also have $B_{q}(a,b) \leq C\log q$ (when $a \neq b$) by \eqref{PosImpliesSmall}, and therefore we may take $\epsilon_{1} \asymp 1/\varphi(q)$, which is very small. It is more difficult to determine a permissible value of $\epsilon_{2}$, but we may do so using Lemma \ref{DistribCorr} together with Correlation Estimate 1. Indeed, we can write $\sum_{l=1}^{k} = \sum_{a \leq 2\log k} \sum_{l \in S_a}$, where $S_a := \{1 \leq l \leq k : \frac{1}{2^a \log q} < |r_{i,l}| \leq \frac{1}{2^{a-1}\log q}\}$, and write $\sum_{s=1}^{k} = \sum_{b \leq 2\log k} \sum_{s \in T_b}$, where $T_b := \{1 \leq s \leq k : \frac{1}{2^b \log q} < |r_{j,s}| \leq \frac{1}{2^{b-1}\log q}\}$ (with suitable adjustments to the definitions of $S_1, S_{\lfloor 2\log k \rfloor}, T_1, T_{\lfloor 2\log k \rfloor}$ to ensure that all indices $l,s$ are included). Then Lemma \ref{DistribCorr} implies that $\#S_a \ll 2^{2a} a^4$ and $\#T_b \ll 2^{2b} b^4$, and using Correlation Estimate 1 to bound all the sums of $|r_{l,s}|$ over $l \in S_a$ and $s \in T_b$, we obtain that
$$ \sum_{l=1}^{k} \sum_{s=1, s \neq l}^{k} |r_{i,l}| |r_{j,s}| |r_{l,s}| \ll \log^{2}(2k) \frac{\log^{4}(2k)}{\log^{2}q} \frac{\log^{2}(2k)}{\log q} = \frac{\log^{8}(2k)}{\log^{3}q} \ll \frac{(\log\log q)^{8}}{\log^{3}q} . $$
By dividing up the sum over $t$ according to the size of $|r_{l,t}|$, one can similarly show that
$$ \sum_{t = 1, t \neq l,s}^{k} |r_{l,t}||r_{s,t}| \ll \log(2k) \frac{\log^{2}(2k)}{\log q} \frac{\log^{2}(2k)}{\log q} \ll \frac{(\log\log q)^{5}}{\log^{2}q} . $$
Putting these estimates together with the standard bounds $ \sum_{l=1}^{k} |r_{i,l}| \ll \frac{\sqrt{k} \log^{2}(2k)}{\log q}$ and $\sum_{\substack{1 \leq l,s \leq k, \\ l \neq s}} |r_{l,s}| \ll \frac{k \log^{2}(2k)}{\log q}$, (coming from Correlation Estimate 1), 
one checks that it is permissible to take
$$ \epsilon_{2} \asymp \frac{(\log\log q)^{9}}{\log^{3} q} , $$
whereupon Lemma \ref{lemforfirstk} implies that
\begin{align*}
 & \p(\max_{k+1 \leq i \leq n} Z_{i} \leq x_k | Z_{1}=x_1 , ..., Z_{k} = x_k) - \prod_{i=k+1}^{n} \Phi\left(w_i \right) \nonumber \\
& \ll  e^{-\left(1+O\left(\sqrt{k/\log q}\right)\right)x_k^2}\left(\prod_{i=k+1}^{n} \Phi\left( \tilde{w_i} \right) + e^{-\Theta\left(A^2 (\log q)/(\log\log q)^{9}\right)} \right) \left( \sum_{k+1 \leq i < j \leq n} |r_{i,j}| + \sum_{l=1}^{k} \left(\sum_{i=k+1}^{n} |r_{i,l}| \right)^2 \right),
\end{align*}
where $\tilde{w_i}=(1+O(1/\log q))w_i.$

Finally, the stated form of the result follows by using Correlation Estimate 1 again to show that the second bracket on the right hand side is
$$ \ll \frac{n \log^{2}n}{\log q} + \frac{nk \log^{4}n}{\log^{2}q} \ll \frac{n \log^{4}n}{\log q}.$$

\end{proof}

\vspace{12pt}
We have now collected all the necessary ingredients for the proof of Theorem \ref{OrderFirstk1}.

\begin{proof}[Proof of Theorem \ref{OrderFirstk1}]
Let $c > 0$ be a small constant to be fixed later, and suppose first that $c n/\log n < k \leq n$. Then our condition $k \log^{10}k \leq (\log q)/\log n$ implies that $n \ll_{c} (\log q)/(\log\log q)^{10}$, in which case Theorem \ref{FullRace} is applicable and implies that
$$ \delta_k(q;a_1, \dots, a_n)= \frac{(n-k)!}{n!}\left(1+O\left(\frac{n(\log n)^4}{\log q} \right)\right) = \frac{(n-k)!}{n!}\left(1+O\left(\frac{k(\log k)^4 \log n}{\log q} \right)\right) . $$
This result is already acceptable for Theorem \ref{OrderFirstk1}, so we may assume henceforth that we are in the other case where $2 \leq k \leq c n/\log n$. We may also assume throughout that $n \geq \log^{0.1}q$, because otherwise the result again follows from Theorem \ref{FullRace} regardless of the value of $k$. Since we assume that $q$ is large, we may assume in particular that $n$ is large.

First, it follows from Normal Approximation Result 1 together with the discussion in section \ref{pnrsetup} that 
$$
 \big |\delta_k(q; a_1,\dots,a_n)- \p(Z_{1} > Z_{2} > ... > Z_{k} > \max_{k+1 \leq i \leq n} Z_{i})\big| \ll \frac{n^{3}}{\varphi(q)^{1/8}}.
$$
This error is completely negligible for Theorem \ref{OrderFirstk1}, since we assume $k \log^{10}k \leq (\log q)/\log n$ and $(\log q)/\log n$ is large.

Let $1\leq A\leq 2\sqrt{\log n}$ be a real parameter to be chosen later. Then we have
\begin{eqnarray}
&& \p(Z_{1} > Z_{2} > ... > Z_{k} > \max_{k+1 \leq i \leq n} Z_{i}) \nonumber \\
& = & \p(Z_{1} > ... > Z_{k} > \max_{k+1 \leq i \leq n} Z_{i}, \; \textrm{and} \; Z_{k} > A) + O\left(\p\left(\max_{k+1 \leq i \leq n} Z_{i} \leq A\right)\right) , \nonumber
\nonumber
\end{eqnarray}
and using  Normal Comparison Result 2 with $\epsilon \asymp 1/\log q$ we get
$$
 \p\left(\max_{k+1 \leq i \leq n} Z_{i} \leq A \right) \ll \exp\left\{- \Theta\left(n \frac{e^{-A^{2}/2 + O(A^2/\log q+AB+B^2)}}{A+B} \right) \right\} + e^{-\Theta(B^2\log q)},
$$
for any $B > 0$. Simply choosing $B=1$, and using that $1 \leq A \leq 2\sqrt{\log n} \leq 2\sqrt{\log q}$ and that $k \log^{10}k \leq (\log q)/\log n$, we deduce 
\begin{equation}\label{eq7.1}
\p\left(\max_{k+1 \leq i \leq n} Z_{i} \leq A \right) \ll \exp\left\{- \Theta\left(n \frac{e^{-A^{2}/2 + O(A)}}{A} \right) \right\} + n^{-2k}.
\end{equation}
Ultimately we will choose $A$ so this whole thing is $\ll n^{-2k} \ll \frac{(n-k)!}{n!} \frac{1}{n \log^{0.1}q}$, which will be an acceptable error term.

Next, we have
\begin{equation}\label{ConditionalProbasEstimates}
\begin{aligned}
& \p(Z_1>Z_2> ... > Z_k>\max_{k+1\leq i \leq n} Z_i, \; \textrm{and} \; Z_{k} > A) \\
= & \int_{x_1>...>x_k > A} f(x_1,..., x_k) \cdot \p\Big(\max_{k+1\leq i \leq n} Z_i <x_k\big| Z_1=x_1, ..., Z_k=x_k\Big) dx_1\cdots dx_k \\
= & \int_{\substack{x_1>...>x_k>A, \\ ||x|| \leq 3\sqrt{k\log n}}} f(x_1, ..., x_k) \cdot \p\Big(\max_{k+1\leq i \leq n} Z_i <x_k\big| Z_1=x_1, ..., Z_k=x_k\Big)dx_1\cdots dx_k + O(n^{-2k}),
\end{aligned}
\end{equation}
where as before $f$ denotes the joint density of $Z_1 , ..., Z_k$, and where the second equality follows because by Proposition \ref{Density} (and the fact that $k \ll \log q$) we always have $ f(x_1 , ..., x_k) \ll \frac{1}{(2\pi)^{k/2}} \exp\left(-\frac{||x||^2}{4}\right)$. Moreover, we note that when $||x|| \leq 3\sqrt{k\log n}$, Proposition \ref{Density} implies
\begin{equation}\label{asympdensityk}
 f(x_1 , ..., x_k) = \left(1 + O\left(k \log^4k \frac{\log n}{\log q}\right)\right) \frac{1}{(2\pi)^{k/2}} \exp\left(-\frac{||x||^2}{2}\right). 
 \end{equation}

Now let us note that  $ \Phi(y\pm\epsilon)= \big(1+O\big( \epsilon e^{-(y+O(\epsilon))^2/2}\big)\big) \Phi(y),$ for any $y, \epsilon>0$. In view of this, and the crude bound $\sum_{l=1}^{k} x_l |r_{i,l}| \ll (1/\log q) \sqrt{k} \sqrt{\sum_{l=1}^{k} x_l^2} \ll \sqrt{k/\log q}$, the main term $\prod_{i=k+1}^{n} \Phi\left(w_i \right)$ in our Lemma \ref{lemfirstkPNR} estimate for $\p\Big(\max_{k+1\leq i \leq n} Z_i <x_k\big| Z_1=x_1, ..., Z_k=x_k\Big)$  equals
\begin{equation}\label{eq7.4}
\begin{aligned}
& \quad \prod_{i=k+1}^{n} \Phi\left((1+O(1/\log q ))x_k + O\Big(\sum_{l=1}^{k} x_l |r_{i,l}| \Big) \right) \\
&=  \prod_{i=k+1}^{n} \Phi\Big(\big(1+O(1/\log q)\big)x_k \Big) \prod_{\substack{i=k+1, \\ \sum_{l=1}^{k} x_l |r_{i,l}| \geq \frac{x_k}{\log q}}}^{n} \left(1 + O\left(e^{-\big(1/2+O(\sqrt{k/\log q})\big)x_k^2} \sum_{l=1}^{k} x_l |r_{i,l}|\right) \right) \\
& =  \Phi\Big(\big(1+O(1/\log q)\big)x_k \Big)^{n-k} \exp\left\{ O\Bigg(e^{-\frac{A^2}{2}\Big(1+O\big(\sqrt{k/\log q}\big)\Big)} \sum_{\substack{i=k+1, \\ \sum_{l=1}^{k} x_l |r_{i,l}| \geq \frac{x_k}{\log q}}}^{n}  \sum_{l=1}^{k} x_l |r_{i,l}| \Bigg) \right\}.
\end{aligned}
\end{equation}
Now for any given $x_1 > x_2 > ... > x_k > A$ satisfying $||x|| \leq 3\sqrt{k\log n}$, and any non-empty subset $\mathcal{S} \subseteq \{k+1,k+2,...,n\}$, we can divide the  sum $\sum_{l=1}^{k} x_l |r_{i,l}|$ into $O(\log(2k))$ pieces depending on the size of $x_l$ (on dyadic ranges), and apply Correlation Estimate 1 (similarly as in the proof of Proposition 5.1) to deduce that
$$ \sum_{i \in \mathcal{S}} \sum_{l=1}^{k} x_l |r_{i,l}| \ll \log(2k) ||x|| \sqrt{\#\mathcal{S}} \frac{\log^{2}(2k\#\mathcal{S})}{\log q} \ll \sqrt{(\#\mathcal{S}) k \log n} \frac{\log^{3}(2k\#\mathcal{S})}{\log q} . $$
But if $\mathcal{S} := \{k+1 \leq i \leq n : \sum_{l=1}^{k} x_l |r_{i,l}| \geq \frac{x_k}{\log q} \}$ then the left hand side must also be $\geq \#\mathcal{S} \frac{x_k}{\log q} \gg \#\mathcal{S} \frac{1}{\log q}$, so we deduce that 
$\#\mathcal{S} \ll k (\log n)\log^{6}\big(2k\log n\big)\ll k (\log n)(\log\log q)^6$. Substituting back this size estimate for $\mathcal{S}$ implies that
$$ \sum_{\substack{i=k+1, \\ \sum_{l=1}^{k} x_l |r_{i,l}| \geq \frac{x_k}{\log q}}}^{n} \sum_{l=1}^{k} x_l |r_{i,l}| \ll  \frac{k  (\log n)(\log\log q)^6}{\log q}. $$
Collecting the above estimates shows that
$$ \prod_{i=k+1}^{n} \Phi\left(w_i \right) = \Phi\left((1+O(1/\log q))x_k \right)^{n-k} \exp\left\{ O\Bigg(e^{-\frac{A^2}{2}\Big(1+O\big(\sqrt{k/\log q}\big)\Big)} \frac{k  (\log n)(\log\log q)^6}{\log q} \Bigg) \right\}, $$
and clearly we have the same estimate for $\prod_{i=k+1}^{n} \Phi\left((1+O(1/\log q))w_i \right)$, which appears in the error term of Lemma \ref{lemfirstkPNR}. 
 

At this point we choose $A$ such that 
\begin{equation}\label{ChoiceA}
e^{0.51 A^{2}} = \frac{n}{k \log n} .
\end{equation}
Note that since we assume that $2 \leq k \leq cn/\log n$, this choice of $A$ will satisfy $1 \leq A \leq 2\sqrt{\log n}$ provided $c$ is fixed small enough. Notice also that if $n \geq \log^2 q$, then $k\log n \ll \log q \leq \sqrt{n}$ and so $A \gg \sqrt{\log n}$. Substituting our choice of $A$ back into \eqref{eq7.1} yields
\begin{equation}\label{eq7.6}
\p\left(\max_{k+1 \leq i \leq n} Z_{i} \leq A \right) \ll  \exp\left\{- \Theta\left(\frac{e^{0.01A^2 + O(A)}}{A} k\log n\right) \right\} + n^{-2k} \ll n^{-2k} ,
\end{equation}
provided $c$ was chosen small so $A$ is large enough. With this choice of $A$ we also have
$$ e^{-A^2 \Big(1+O\big(\sqrt{k/\log q}\big)\Big)}\frac{n(\log n)^4}{\log q} = e^{-A^2 \Big(0.49+O\big(\sqrt{k/\log q}\big)\Big)}\frac{k(\log n)^5}{\log q} \ll \frac{k(\log k)^4 \log n}{\log q} , $$
and similarly (distinguishing cases according as $k \geq \sqrt{n} \geq \log^{0.05}q$ or not) that
$$ e^{-\frac{A^2}{2}\Big(1+O\big(\sqrt{k/\log q}\big)\Big)} \frac{k  (\log n)(\log\log q)^6}{\log q} \ll \frac{k  (\log n) (\log k)^6}{\log q} \ll 1 . $$
Substituting into Lemma \ref{lemfirstkPNR}, and using our previous computations of $\prod_{i=k+1}^{n} \Phi\left(w_i \right)$ and our assumption that $k \ll (\log q)/(\log\log q)^{10}$, yields that
\begin{equation}\label{approxprobas}
\p(\max_{k+1 \leq i \leq n} Z_{i} \leq x_k | Z_{1}=x_1 , ..., Z_{k} = x_k)=  \Phi\big((1+O(\frac{1}{\log q}))x_k \big)^{n-k} \left(1+O\left(\frac{k \log n \log^{6}k}{\log q}\right) \right) + O(n^{-2k}) .
\end{equation}

Now, in view of the estimates \eqref{asympdensityk} and \eqref{approxprobas}, the main term on the right hand side of \eqref{ConditionalProbasEstimates} equals
\begin{eqnarray}
&& \int_{\substack{x_1>\cdots>x_k>A, \\ ||x|| \leq 3\sqrt{k\log n}}} \frac{1}{(2\pi)^{k/2}} e^{-||x||^{2}/2} \Phi\left((1+O(1/\log q))x_k \right)^{n-k} dx_1 ... dx_k \nonumber \\
& = & \int_{x_1>\cdots>x_k>A} \frac{1}{(2\pi)^{k/2}} e^{-||x||^{2}/2} \Phi\left((1+O(1/\log q))x_k \right)^{n-k} dx_1 ... dx_k + O(n^{-2k}) . \nonumber
\end{eqnarray}
Thus Theorem \ref{OrderFirstk1} will certainly follow if we show that
\begin{equation}\label{MainIntProb} 
\int_{x_1>\cdots>x_k>A} \frac{1}{(2\pi)^{k/2}} e^{-||x||^{2}/2} \Phi\left((1+O(1/\log q))x_k \right)^{n-k} dx_1 ... dx_k = \left(1+O\left(\frac{k\log n}{\log q}\right)\right)\frac{(n-k)!}{n!}.
\end{equation}
We shall only prove an upper bound, since an exactly similar argument would give a matching lower bound. Indeed, for a certain absolute constant $C > 0$ the integral is
\begin{eqnarray}
& \leq & \int_{x_1>\cdots>x_k>A} \frac{1}{(2\pi)^{k/2}} e^{-||x||^{2}/2} \Phi\left((1+ C/\log q)x_k \right)^{n-k} dx_1 \cdots dx_k \nonumber \\
& = & \int_{x_k > A} \frac{e^{-x_{k}^{2}/2}}{\sqrt{2\pi}} \Phi\left((1+C/\log q)x_k \right)^{n-k} \frac{1}{(k-1)!} \left(\int_{x > x_k} \frac{e^{-x^{2}/2}}{\sqrt{2\pi}} dx \right)^{k-1} dx_k, \nonumber
\end{eqnarray}
by symmetry of the integration variables $x_1, \dots, x_{k-1}$.
Making a substitution shows this is
\begin{eqnarray}
& = & \left(1+O\left(\frac{1}{\log q}\right)\right) \int_{x_k > (1+\frac{C}{\log q}) A} \frac{e^{-(1+O(1/\log q))x_k^{2}/2}}{\sqrt{2\pi}} \Phi\left( x_k \right)^{n-k}  \nonumber \\
&&  \quad \quad \quad \quad \quad \quad\quad \quad \quad \quad \times \frac{1}{(k-1)!} \left(\int_{x > (1+C/\log q)^{-1} x_k} \frac{e^{-x^{2}/2}}{\sqrt{2\pi}} dx \right)^{k-1} dx_k \nonumber \\
& = & \left(1+O\left(\frac{1}{\log q}\right)\right) \int_{x_k > (1+\frac{C}{\log q})A} \frac{e^{-x_k^{2}/2}}{\sqrt{2\pi}} \Phi\left( x_k \right)^{n-k} \cdot \frac{e^{O(k x_k^{2}/\log q)}}{(k-1)!} \left(\int_{x > x_k} \frac{e^{-x^{2}/2}}{\sqrt{2\pi}} dx \right)^{k-1} dx_k , \nonumber
\end{eqnarray}
and here the term $e^{O(k x_k^{2}/\log q)}$ is $e^{O(k \log n/\log q)} = 1+O(\frac{k \log n}{\log q})$ provided $x_k \leq 3\sqrt{\log n}$, say. Moreover (and as seen before), the contribution to the integral from the complementary range $x_k > 3\sqrt{\log n}$ is $\ll n^{-2k}$, since in this case we have $||x||\geq 3\sqrt{k\log n}$. Hence, our integral equals
\begin{align*}
&\left(1+O\left(\frac{k \log n}{\log q}\right)\right)\int_{x_k > (1+\frac{C}{\log q})A} \frac{e^{-x_k^{2}/2}}{\sqrt{2\pi}} \Phi\left( x_k \right)^{n-k} \cdot \frac{1}{(k-1)!} \left(\int_{x > x_k} \frac{e^{-x^{2}/2}}{\sqrt{2\pi}} dx \right)^{k-1} dx_k + O(n^{-2k})\\
&=\left(1+O\left(\frac{k \log n}{\log q}\right)\right)\int_{x_1>\cdots>x_k>(1+\frac{C}{\log q})A} \frac{1}{(2\pi)^{k/2}} e^{-||x||^{2}/2} \Phi\left(x_k \right)^{n-k} dx_1\cdots  dx_k + O(n^{-2k})\\
&= \left(1+O\left(\frac{k \log n}{\log q}\right)\right)\cdot \p\left(\tilde{Z}_{1} > \tilde{Z}_{2} > ... > \tilde{Z}_{k} > \max_{k+1 \leq i \leq n} \tilde{Z}_{i}, \; \textrm{and} \; \tilde{Z}_{k} > (1+\frac{C}{\log q})A\right) + O(n^{-2k})
\end{align*}
where the $\tilde{Z}_{i}$ are independent standard normal random variables. Furthermore, the same argument leading to \eqref{eq7.6}  shows that 
$$ \p\left( \max_{k+1 \leq i \leq n} \tilde{Z}_{i}< (1+\frac{C}{\log q})A\right) \ll n^{-2k}. $$
Finally, the asymptotic \eqref{MainIntProb}  follows from combining the above estimates, and using that the probability $\p(\tilde{Z}_{1} > \tilde{Z}_{2} > ... > \tilde{Z}_{k} > \max_{k+1 \leq i \leq n} \tilde{Z}_{i})$ equals  $(n-k)!/n!$ by symmetry of the random variables $\tilde{Z}_{1},  \tilde{Z}_{2}, \dots, \tilde{Z}_{n}$.

\end{proof}

\section{Irregularities in the densities: Proof of Theorem \ref{OrderFirstk2}}\label{secdeviation}
As in the previous section, given distinct reduced residues $a_1, \dots, a_n$ mod $q$ we let $Z=(Z_{j})_{1 \leq j \leq n}$ be a multivariate normal random vector whose components have mean zero, variance one, and correlations 
$r_{i,j}=\E Z_{i}Z_{j} := \frac{B_{q}(a_i,a_j)}{\var(q)}.$
Also, we let $2 \leq k\leq n$ be a fixed positive integer and let $f(x_1, \dots, x_k)$ denote the density function of $Z_1, \dots, Z_k$. We denote by  $\mathcal{C}=(r_{i,j})_{1\leq i,j\leq k}$ the covariance matrix of $Z_1, \dots, Z_k$ (which is certainly invertible provided $q$ is large enough in terms of $k$), and let $\widetilde{r}_{i,j}$ denote the entries of the inverse matrix $\mathcal{C}^{-1}$.

The key ingredient in the proof of Theorem \ref{OrderFirstk2} is the following result.

\begin{lem}\label{NonStandardDensity}
Let $k\geq 2$ be fixed, and let $q$ be large enough in terms of $k$. There exist distinct reduced residues $a_1, \dots, a_k$ modulo $q$ such that 
\begin{align*}
&f(x_1, x_2, \dots, x_k)\\
&= \left(1+O_k\left(\frac{1}{(\log q)^2}\right)\right)\frac{1}{(2\pi)^{k/2}}\exp\left(-\frac{||x||^2}{2}-x_1x_2\frac{\log 2+o(1)}{\log q} +O_k\left(\frac{||x||^2}{(\log q)^2}\right)\right) ,
\end{align*}
where the $o(1)$ term tends to zero as $q \rightarrow \infty$.
\end{lem}

\begin{proof}
Let $p_1<p_2$ be the smallest prime numbers such that $(p_1p_2, q)=1$. Then one has
$p_1<p_2\leq 2\log q$ in view of the fact that $\prod_{p\leq z} p= e^{z(1+o(1))}$, which follows
from the prime number theorem. Mimicking a construction used in Theorem 2 of Lamzouri~\cite{La2}, we let $a_1=1$, $a_2=-1$ and $a_j= (p_1p_2)^j$ for $3\leq j\leq k$. Then by Lemma \ref{CovarianceEntries} and equation \eqref{AsympVariance}, it follows that 
\begin{equation}\label{LARGECor}
r_{1, 2}=r_{2,1}=-\frac{(\log 2)+o(1)}{\log q},
\end{equation}
and
\begin{equation}\label{SMALLCor}
r_{i,j}\ll \frac{(2\log q)^{2k+1}}{\varphi(q)},
\end{equation}
for all $i\neq j$ such that $\{i,j\}\neq \{1,2\}$. Now, by \eqref{DefDensity} we have 
$$f(x_1, \dots, x_k)=  \left(1+O_k\left(\frac{1}{(\log q)^2}\right)\right)\frac{1}{(2\pi)^{k/2}}\exp\left(-\frac12 \mathbf{x}^T \mathcal{C}^{-1} \mathbf{x}\right).$$
Recall that $\widetilde{r}_{i,i}=1+O_k(1/(\log q)^2)$ by \eqref{estimateR}. Furthermore, by \eqref{estimateR} and \eqref{SMALLCor} we have 
$$\widetilde{r}_{i,j}\ll_k \frac{1}{(\log q)^2},$$
for all $i\neq j$ such that $\{i, j\}\neq \{1, 2\}.$ This implies
\begin{equation}\label{BiasedDensity}
\begin{aligned}
&f(x_1, \dots, x_k)\\
&=  \left(1+O_k\left(\frac{1}{(\log q)^2}\right)\right)\frac{1}{(2\pi)^{k/2}}\exp\left(-\frac{||x||^2}{2}-x_1x_2 \frac{ \widetilde{r}_{2,1}+ \widetilde{r}_{1,2}}{2}+O_k\left(\frac{||x||^2}{(\log q)^2}\right)\right).
\end{aligned}
\end{equation}

Now, let $\mathcal{A}$ be the matrix obtained from $\mathcal{C}$ by removing the first row and the second column. Then arguing as in the proof of Lemma \ref{matrixlem}, we obtain
\begin{eqnarray}
\widetilde{r}_{2,1}= -\frac{\det(\mathcal{A})}{\det(\mathcal{C})} & = &- \left(1+O_k\left(\frac{1}{(\log q)^2}\right)\right)\det(\mathcal{A}) \nonumber \\
& = & - \left(1+O_k\left(\frac{1}{(\log q)^2}\right)\right) \left(r_{2,1}+ O_k\left(\frac{(2\log q)^{2k+1}}{\varphi(q)}\right) \right) . \nonumber
\end{eqnarray}
Thus by \eqref{LARGECor} we obtain
$ \widetilde{r}_{2,1}=\frac{\log 2+o(1)}{\log q},$ and
a similar estimate holds for $\widetilde{r}_{1,2}$. Inserting these estimates in \eqref{BiasedDensity} completes the proof. 
\end{proof}

\begin{proof}[Proof of Theorem \ref{OrderFirstk2}]
By the same argument that was used in the Introduction to deduce Theorem \ref{BigBias} from Theorem \ref{OrderFirstk2}, it will suffice to prove Theorem \ref{OrderFirstk2} in the case $k=2$, since this implies the result for larger $k$. (Using Lemma \ref{NonStandardDensity} one could in fact prove the result directly for any fixed $k$, but this would be more complicated and require an unwanted additional assumption of the shape $n < \varphi(q)^{c/k}$ at the end.)

Let $a_1 = 1$ and $a_2 = -1$ be as in Lemma \ref{NonStandardDensity}, and let $a_{3}, \dots, a_n$ be distinct reduced residues modulo $q$ that are different from $a_1, a_2$. As in the proof of Theorem \ref{OrderFirstk1}, Normal Approximation Result 1 implies that 
$$
 \big |\delta_2(q; a_1,\dots,a_n)- \p(Z_{1} > Z_{2} > \max_{3 \leq i \leq n} Z_{i})\big| \ll \frac{n^{3}}{\varphi(q)^{1/8}} .
$$

Next, we have
\begin{eqnarray*}
& &\p(Z_{1} > Z_{2} > \max_{3 \leq i \leq n} Z_{i})  \\
& = & \p(Z_{1} > Z_{2} > \max_{3 \leq i \leq n} Z_{i}, \; \textrm{and} \; Z_{2} > \sqrt{1.99\log n}) + O(\p(\max_{3 \leq i \leq n} Z_{i} \leq \sqrt{1.99\log n})), 
\end{eqnarray*}
and using Normal Comparison Result 2 we obtain
\begin{equation}\label{boundtailproba}
 \p(\max_{3 \leq i \leq n} Z_{i} \leq \sqrt{1.99\log n}) \ll e^{-n^{c_1}} + n^{-c_2\log q} \ll n^{-4},
 \end{equation}
for some positive constants $c_1, c_2$. Moreover, similarly to \eqref{ConditionalProbasEstimates} we derive
\begin{equation}\label{CutFinitek}
\begin{aligned}
& \p(Z_{1} > Z_{2} > \max_{3 \leq i \leq n} Z_{i}, \; \textrm{and} \; Z_{2} > \sqrt{1.99\log n})\\
= & \int_{\substack{x_1>x_2>\sqrt{1.99\log n}\\ ||x||<3\sqrt{2\log n}}}f(x_1, x_2) \cdot \p\Big(\max_{3\leq i \leq n} Z_i <x_2 \big| Z_1=x_1, Z_2=x_2\Big)dx_1 dx_2 + O\left(n^{-4}\right).\\
\end{aligned}
\end{equation}

Next, it follows from Lemma \ref{lemfirstkPNR} that for all $x= (x_1,x_2)$ such that $x_1>x_2>\sqrt{1.99\log n}$ and $||x||<3\sqrt{2\log n}$ we have
\begin{equation}\label{approxcondi3}
\begin{aligned}
&\p(\max_{3 \leq i \leq n} Z_{i} \leq x_2 | Z_{1}=x_1 , Z_{2} = x_2)
= \prod_{i=3}^{n} \Phi\left(w_i \right)\\
& \quad \quad \quad + O\left(\frac{1}{n^{9/10}\log q} \left(\prod_{i=3}^{n} \Phi\left((1+O(1/\log q))w_i \right) + n^{-4} \right) \right), 
\end{aligned}
\end{equation}
where $w_i= (1+O(1/\log^2q)) x_2+ O(||x||\sum_{s=1}^2 |r_{i, s}|)$. Then, similarly to \eqref{eq7.4} one gets 
\begin{align*}
 \prod_{i=3}^{n} \Phi\left(w_i \right)
 &= \prod_{i=3}^{n} \Phi\left((1+O(1/\log^2q)) x_2\right)
 \left(1 + O\left(||x|| \cdot e^{-\big(1/2+O(\sqrt{1/\log q})\big)x_2^2}\sum_{i=3}^n \sum_{s=1}^2 |r_{i,s}| \right) \right)\\
 &=  \Phi\left((1+O(1/\log^2q)) x_2\right)^{n-2} \left(1+O\left(\frac{1}{n^{2/5}\log q}\right)\right),
 \end{align*}
 by Correlation Estimate 1 and the fact that $x_2 > \sqrt{1.99\log n}$. Moreover, in the same range for the $x_i$ we deduce from Lemma \ref{NonStandardDensity} that
\begin{align*}
 f(x_1, x_2) &= \left(1+O\left(\frac{\log n}{(\log q)^2}\right)\right)\frac{1}{(2\pi)^{2/2}}\exp\left(-\frac{||x||^2}{2}-x_1x_2\frac{\log 2+o(1)}{\log q}\right)\\
 &\leq \left(1-\frac{c\log n}{\log q}\right) \frac{1}{2\pi}e^{-||x||^2/2}
\end{align*}
for some positive constant $c>0$, provided $q$ is large enough. 

Combining the above estimates, and using our assumption that $n \geq \varphi(q)^{\epsilon}$, we deduce that the main term in \eqref{CutFinitek} equals
\begin{align*}
 &\int_{\substack{x_1> x_2>\sqrt{1.99\log n}\\ ||x||<3\sqrt{2\log n}}}f(x_1, x_2) \cdot  \prod_{i=3}^{n} \Phi\left(w_i \right)dx_1 dx_2\\
 & \leq \left(1-c_{\epsilon}\right) \int_{\substack{x_1>x_2>\sqrt{1.99\log n}\\ ||x||<3\sqrt{2\log n}}}\frac{1}{2\pi}e^{-||x||^2/2} \Phi\left((1+O(1/\log^2q)) x_2\right)^{n-2}dx_1 dx_2,
\end{align*}
for some positive constant $c_{\epsilon}$.
Furthermore, using an exactly similar argument as in the proof of \eqref{MainIntProb}, together with \eqref{boundtailproba}, we derive
\begin{align*}
 &\int_{\substack{x_1>x_2>\sqrt{1.99\log n}\\ ||x||<3\sqrt{2\log n}}}\frac{1}{2\pi}e^{-||x||^2/2} \Phi\left((1+O(1/\log^2q)) x_2\right)^{n-2}dx_1 dx_2 \\
 &= \left(1+O\left(\frac{\log n}{\log^2 q}\right)\right) \frac{(n-2)!}{n!}= \left(1+O\left(\frac{1}{\log q}\right)\right) \frac{(n-2)!}{n!}.
\end{align*}
Finally, the total contribution of the various error terms (notably the one from Normal Approximation Result 1 and the one from \eqref{approxcondi3}) is 
$$ \ll \frac{n^3}{\varphi(q)^{1/8}} + n^{-4} + \frac{1}{n^{0.9}\log q} \frac{(n-2)!}{n!}. $$
Recalling our assumption that $\varphi(q)^{\epsilon} \leq n < \varphi(q)^{1/41}$, we see the error is negligible compared with $\frac{(n-2)!}{n!}$, which completes the proof. 

\end{proof}

\appendix

\section{Sketch proof of Lemma \ref{harmanal}}\label{analysisappendix}
In this appendix we sketch the proof of Lemma \ref{harmanal}, the harmonic analysis lemma that we used to prove Correlation Estimate 1. Lemma \ref{harmanal} is inspired by a result in work of Bourgain~\cite{bourgain} (in a substantially different context), but differs in that it is concerned with two sets of points $\theta_{r}$ and $\phi_{s}$ rather than one, and it involves a weight $\Lambda(q)/q$ rather than $1/q$. It turns out that the former adaptation is easy, and the latter one simplifies and strengthens the argument (since all work with divisor functions becomes much easier). For the sake of completeness, we provide a fairly full sketch proof of Lemma \ref{harmanal} here.

Throughout we let $\textbf{1}$ denote the indicator function, and let $|| \cdot ||$ denote distance mod 1. Recall that we are given two sets $\theta_{1}, ..., \theta_{R}$ and $\phi_{1}, ..., \phi_{S}$ of $1/x$-spaced real numbers. Corresponding to these, let $I_{\theta}, I_{\phi} : \R/\Z \rightarrow [0,10]$ be bounded variation continuous functions on the real numbers mod 1, that satisfy
$$ I_{\theta}(t) \geq \sum_{1 \leq r \leq R} \textbf{1}_{|| t - \theta_{r}|| \leq 1/x} \;\;\; \text{and} \;\;\; I_{\phi}(t) \geq \sum_{1 \leq s \leq S} \textbf{1}_{|| t - (-\phi_{s})|| \leq 1/x} , $$
(note the negative signs attached to the $\phi_{s}$ here), and also
$$ \int_{0}^{1} I_{\theta}(t) dt \leq \frac{1000R}{x}, \;\;\; \int_{0}^{1} I_{\phi}(t) dt \leq \frac{1000S}{x}, \;\;\; \text{supp}(\hat{I_{\theta}}), \; \text{supp}(\hat{I_{\phi}}) \subseteq [-x,x] , $$
where $\hat{I_{\theta}}(k) := \int_{0}^{1} I_{\theta}(t) e^{-2\pi i tk} dt$ and $\hat{I_{\phi}}(k)$ (for $k \in \Z$) denote the Fourier transforms of $I_{\theta}, I_{\phi}$. It is a standard fact that one can construct such functions $I$ (e.g. as a sum of Beurling--Selberg type smooth functions that approximate $\textbf{1}_{||t - \theta_{r}|| \leq 1/x}$ and $\textbf{1}_{||t - (-\phi_{s})|| \leq 1/x}$), and we see immediately that the convolution
$$ (I_{\theta} \ast I_{\phi})(t) := \int_{0}^{1} I_{\theta}(u) I_{\phi}(t-u) du \geq \sum_{\substack{1 \leq r \leq R, \\ 1 \leq s \leq S}} \int_{0}^{1} \textbf{1}_{||u-\theta_{r}|| \leq \frac{1}{x}} \textbf{1}_{||t-u+\phi_{s}|| \leq \frac{1}{x}} du \geq \frac{1}{x} \sum_{\substack{1 \leq r \leq R, \\ 1 \leq s \leq S}} \textbf{1}_{||t - (\theta_{r} - \phi_{s})|| \leq \frac{1}{x}} . $$

Consequently, in Lemma \ref{harmanal} we have
$$ \sum_{\substack{1 \leq r \leq R, \\ 1 \leq s \leq S}} G(\theta_{r}-\phi_{s}) = \sum_{q \leq Q} \frac{\Lambda(q)}{q} \sum_{a=0}^{q-1} \sum_{\substack{1 \leq r \leq R, \\ 1 \leq s \leq S}} \textbf{1}_{||(\theta_{r} - \phi_{s}) - a/q|| \leq \frac{1}{x}} \leq x \sum_{q \leq Q} \frac{\Lambda(q)}{q} \sum_{a=0}^{q-1} (I_{\theta} \ast I_{\phi})(a/q) , $$
and on writing $(I_{\theta} \ast I_{\phi})(a/q)$ in terms of its Fourier coefficients we find
$$ \sum_{a=0}^{q-1} (I_{\theta} \ast I_{\phi})(a/q) = \sum_{a=0}^{q-1} \sum_{k=-\infty}^{\infty} \widehat{(I_{\theta} \ast I_{\phi})}(k) e^{2\pi i ka/q} = q \sum_{\substack{k=-\infty, \\ q \mid k}}^{\infty} \widehat{(I_{\theta} \ast I_{\phi})}(k) . $$
Using also the fact that $\widehat{(I_{\theta} \ast I_{\phi})}(k) = \widehat{I_{\theta}}(k) \widehat{I_{\phi}}(k)$, we conclude overall that
$$ \sum_{\substack{1 \leq r \leq R, \\ 1 \leq s \leq S}} G(\theta_{r}-\phi_{s})  \leq x \sum_{q \leq Q} \Lambda(q) \sum_{\substack{k=-\infty, \\ q \mid k}}^{\infty} \widehat{I_{\theta}}(k) \widehat{I_{\phi}}(k) \leq x \sum_{k=-\infty}^{\infty} \sum_{\substack{q \leq Q, \\ q \mid k}} \Lambda(q) |\widehat{I_{\theta}}(k)| |\widehat{I_{\phi}}(k)| . $$

To finish, define $\mathcal{B} := \{-x \leq k \leq x : k \neq 0, \; \sum_{\substack{q \leq Q, \\ q \mid k}} \Lambda(q) \geq \log^{2}(2QRS)\}$, and note that since $\hat{I_{\theta}}, \hat{I_{\phi}}$ are supposed to vanish outside the interval $[-x,x]$, and since we always have $\sum_{\substack{q \leq Q, \\ q \mid k}} \Lambda(q) \leq \sum_{q \leq Q} \Lambda(q) \ll Q$ and also $|\widehat{I_{\theta}}(k)| \leq \int_{0}^{1} I_{\theta}(t) dt \leq \frac{1000R}{x}$ and $|\widehat{I_{\phi}}(k)| \leq \frac{1000S}{x}$, the right hand side above is
\begin{eqnarray}
& \ll & \log^{2}(2QRS) x \sum_{-x \leq k \leq x} |\widehat{I_{\theta}}(k)| |\widehat{I_{\phi}}(k)| + Qx |\widehat{I_{\theta}}(0)| |\widehat{I_{\phi}}(0)| + Q x \sum_{k \in \mathcal{B}} |\widehat{I_{\theta}}(k)| |\widehat{I_{\phi}}(k)| \nonumber \\
& \ll & \log^{2}(2QRS) x \sqrt{ \sum_{k} |\widehat{I_{\theta}}(k)|^2 } \sqrt{ \sum_{k} |\widehat{I_{\phi}}(k)|^2 } + \frac{RSQ}{x} + \frac{RSQ}{x} \#\mathcal{B} . \nonumber
\end{eqnarray}
The second term above is acceptable for Lemma \ref{harmanal}. By Parseval's Identity we have $$\sqrt{ \sum_{k} |\widehat{I_{\theta}}(k)|^2 } \sqrt{ \sum_{k} |\widehat{I_{\phi}}(k)|^2 } \ll \sqrt{\int_{0}^{1} I_{\theta}(t)^2 dt} \sqrt{\int_{0}^{1} I_{\phi}(t)^2 dt} \ll \frac{\sqrt{RS}}{x}$$, and so the first term above is also acceptable. Now since we always have $\sum_{\substack{q \leq Q, \\ q \mid k}} \Lambda(q) \leq \sum_{q \mid k} \Lambda(q) = \log |k|$ provided $k \neq 0$, there can be no elements of $\mathcal{B}$ with modulus less than $\exp\{\log^{2}(2QRS)\}$. Moreover, if $\exp\{\log^{2}(2QRS)\} \leq k \leq x$ and $\sum_{\substack{q \leq Q, \\ q \mid k}} \Lambda(q) \geq \log^{2}(2QRS)$ then we must be able to write $k = nm$, where $n$ is {\em $Q$ smooth} (i.e. all the prime factors of $n$ are at most $Q$) and also $n \geq \exp\{\log^{2}(2QRS)\}$. Therefore we have
$$ \#\mathcal{B} \leq 2 \sum_{\substack{\exp\{\log^{2}(2QRS)\} \leq n \leq x, \\ n \; \text{is} \; Q \; \text{smooth}}} \frac{x}{n} , $$
and standard upper bounds for the counting function of smooth numbers (see e.g. Theorem 7.6 of Montgomery and Vaughan~\cite{mv}) imply the right hand side is $\ll x/(QRS)^{10}$. It follows that $\frac{RSQ}{x} \#\mathcal{B} \ll 1/(QRS)^{9} \leq 1$, which is certainly an acceptable contribution for Lemma \ref{harmanal}.
\qed

\section{Sketch proof of Lemma \ref{reinrollnormal}}\label{normapproxappend}
In this appendix we very briefly indicate how to deduce Lemma \ref{reinrollnormal} from Theorem 2.1 of Reinert and R\"{o}llin~\cite{rr}.

Indeed, the exchangeable pair construction and calculations used to deduce Central Limit Theorem 1, in an appendix of the preprint~\cite{harpertypicalmax}, transfer directly to this situation and imply that
$$ |\E h(W) - \E h(Z)| \ll |h|_2 \sum_{a,b \in \mathcal{A}} \sqrt{\sum_{i=1}^{m} |c_{i}(a)|^{2}|c_{i}(b)|^{2} \E|V_{i}|^{4}} + |h|_3 \sum_{i=1}^{m} \E|V_{i}|^{3} \left(\sum_{a \in \mathcal{A}} |c_{i}(a)| \right)^{3} . $$
In Lemma \ref{reinrollnormal} we assume the uniform fourth moment bound $\E|V_{i}|^4 \leq K^{4}/m^{2}$, and so the first term is
$$ \ll \frac{|h|_2 K^{2}}{m} \sum_{a,b \in \mathcal{A}} \sqrt{\sum_{i=1}^{m} |c_{i}(a)|^{2}|c_{i}(b)|^{2}} , $$
which is acceptable for Lemma \ref{reinrollnormal}. We also note that
$$ \E|V_{i}|^{3} \leq \frac{K^{3}}{m^{3/2}} + \E|V_{i}|^{3} \textbf{1}_{|V_i| > K/\sqrt{m}} \leq \frac{K^{3}}{m^{3/2}} + \frac{\sqrt{m}}{K} \E|V_{i}|^{4} \leq 2\frac{K^{3}}{m^{3/2}}, $$
and therefore the second term above is
$$ \ll \frac{|h|_3 K^{3}}{m^{3/2}} \sum_{i=1}^{m} \left(\sum_{a \in \mathcal{A}} |c_{i}(a)| \right)^{3} , $$
as required for Lemma \ref{reinrollnormal}.
\qed


\begin{thebibliography}{99}

\bibitem{bourgain} J. Bourgain, \emph{On $\Lambda(p)$-subsets of squares.} 
Israel J. Math., \textbf{67} (1989), no. 3,  291-311. 

\bibitem{FeM}   A. Feuerverger and G. Martin,
\emph{Biases in the Shanks-R\'enyi prime number race}.
Experiment. Math. \textbf{9} (2000), no. 4, 535-570.

\bibitem{Fi} D. Fiorilli, \emph{Highly biased prime number races.}
 Algebra Number Theory \textbf{8} (2014), no. 7, 1733-1767.

\bibitem{FiM}  D. Fiorilli and G. Martin,
\emph{Inequities in the Shanks-R\'enyi Prime Number Race: An asymptotic formula for the densities}. 
 J. Reine Angew. Math.  \textbf{676} (2013), 121-212.

\bibitem{FK1} K. Ford and S. Konyagin,
\emph{The prime number race and zeros of $L$-functions off the critical line}. Duke Math. J. \textbf{113} (2002), no. 2, 313-330.

\bibitem{FK2} K. Ford and S. Konyagin,
\emph{Chebyshev's conjecture and the prime number race}. Modern Problems of Number Theory and its Applications; Topical Problems Part II (Tula, Russia, 2001).

\bibitem{FKL} K. Ford, S. Konyagin and Y. Lamzouri, \emph{The prime number race and zeros of Dirichlet L-functions off the critical line: Part III}.\
 Q. J. Math. \textbf{64} (2013), no. 4, 1091-1098. 

\bibitem{GM}  A. Granville and G. Martin,
\emph{Prime number races}. Amer. Math. Monthly \textbf{113} (2006), no. 1, 1-33.

\bibitem{harpertypicalmax} A. J. Harper, \emph{A note on the maximum of the Riemann zeta function, and log-correlated random variables}. Preprint available online at {\verb http://arxiv.org/abs/1304.0677 }

\bibitem{Ka1} J. Kaczorowski,
\emph{A contribution to the Shanks-R\'enyi race problem}.
Quart. J. Math. Oxford Ser. (2) \textbf{44} (1993), no. 176, 451-458.

\bibitem{Ka2} J. Kaczorowski,
\emph{Results on the distribution of primes}.
J. Reine Angew. Math. \textbf{446} (1994), 89-113.

\bibitem{Ka3} J. Kaczorowski,
\emph{On the Shanks-R\'enyi race problem}.
Acta Arith. \textbf{74} (1996), no. 1, 31-46.

\bibitem{KT}  S. Knapowski and P. Tur\'an,
\emph{Comparative prime-number theory. I.}.
 Acta Math. Acad. Sci. Hungar. \textbf{13} (1962) 299-314; II. \textbf{13} (1962), 315-342;
 III. \textbf{13} (1962), 343-364; IV. \textbf{14} (1963), 31-42; V. \textbf{14} (1963), 43-63; VI. \textbf{14} (1963), 65-78; VII. \textbf{14} (1963), 241-250; VIII. \textbf{14} (1963), 251-268.

\bibitem{La1}  Y. Lamzouri, \emph{The Shanks-R\'enyi prime number race with many contestants.} 
Math. Res. Lett. \textbf{19} (2012), no. 03, 649-666.

\bibitem{La2} Y. Lamzouri, \emph{Prime number races with three or more competitors.}
 Math. Ann. \textbf{356} (2013), no. 3, 1117-1162.

\bibitem{lishao} W. Li and Q.-M. Shao, \emph{A normal comparison inequality and its applications.} Probab. Theory Relat. Fields, \textbf{122} (2002),  494-508. 

\bibitem{Li} J. E. Littlewood, \emph{Distribution des nombres premiers.}
 C. R. Acad. Sci. Paris \textbf{158} (1914), 1869-1872.
 
 \bibitem{MS} G. Martin and J. Scarfy, 
 \emph{Comparative prime number theory: A survey}. 37 pages. {\verb arXiv:1202.3408 }

\bibitem{mv} H. L. Montgomery and R. C. Vaughan, {\em Multiplicative Number Theory I: Classical Theory.} First edition, published by Cambridge University Press, 2007.

\bibitem{pit} V. Piterbarg, {\em Asymptotic Methods in the Theory of Gaussian Processes and Fields.} Volume 148 of the American Mathematical Society Translations of Mathematical Monographs, 1996.

\bibitem{rr} G. Reinert and A. R\"{o}llin, \emph{Multivariate Normal Approximation with Stein's Method of 
Exchangeable Pairs under a General Linearity Condition.} Ann. Probab., \textbf{37} (2009),
no. 6,  2150-2173. 

\bibitem{RuSa}  M. Rubinstein and P. Sarnak,
\emph{Chebyshev's bias}.
Experiment. Math. \textbf{3} (1994), no. 3, 173-197.


\end{thebibliography}
\end{document}